\documentclass[12pt, a4paper]{amsart}%
\usepackage{amscd}
\usepackage{amsmath}
\usepackage{latexsym}
\usepackage{graphicx}
\usepackage{amsfonts}
\usepackage{amssymb}%
\providecommand{\U}[1]{\protect\rule{.1in}{.1in}}
\providecommand{\U}[1]{\protect\rule{.1in}{.1in}}
\nonstopmode
\newcounter{fig}

\theoremstyle{plain}
\newtheorem{Theorem}{Theorem}[section]

\newtheorem{Remarks}[Theorem]{Remarks}

\newtheorem{theorem}[Theorem]{Theorem}

\newtheorem{corollary}[Theorem]{Corollary}
\newtheorem{proposition}[Theorem]{Proposition}
\newtheorem{lemma}[Theorem]{Lemma}
\theoremstyle{definition}

\newtheorem{defn}[equation]{Definition}
\newtheorem{remark}[Theorem]{Remark}
\theoremstyle{remark}

\def\Changed/{\ifvmode\else\vadjust{\vbox to 0pt{\vskip -\baselineskip\hbox to 0pt{\hss\vrule height 0pt depth 1.2\baselineskip\hskip 1em}\vss}}\fi}
\def\Math#1{\def\MathString{#1}\futurelet\MathDelim\MathChoose}
\def\MathChoose{\ifmmode\let\MathDo\MathString              \else\let\MathDo\MathSkip\fi              \MathDo}
\def\MathSkip{\ifx\MathDelim/\def\MathDo{$\MathString$\EatOne}              \else\def\MathDo{$\MathString$}\fi              \MathDo}
\def\Text#1{\def\TextString{#1}\futurelet\TextDelim\TextSkip}
\def\TextSkip{\ifx\TextDelim/\def\TextDo{\TextString\EatOne}              \else\let\TextDo\TextString\fi              \TextDo}
\def\EatOne#1{}
\def\SkipToEndScan#1\EndScan{}
\def\Scan#1#2#3{\ifx#1#2#3\expandafter\SkipToEndScan\fi\Scan#1}
\def\Upper#1{\Scan#1aAbBcCdDeEfFgGhHiIjJkKlLmMnNoOpPqQrRsStTuUvVwWxXyYzZ#1#1\EndScan}
\def\Phrase#1 #2/#3/#4=#5 #6/#7/#8.{\expandafter\edef\csname#2#3\endcsname{\noexpand\Text{#6#7}}
\expandafter\edef\csname\Upper#2#3\endcsname{\noexpand\Text{\Upper#6#7}}
\expandafter\edef\csname#1#2#3\endcsname{\noexpand\Text{#5 #6#7}}
\expandafter\edef\csname\Upper#1#2#3\endcsname{\noexpand\Text{\Upper#5 #6#7}}
\expandafter\edef\csname#2#4\endcsname{\noexpand\Text{#6#8}}
\expandafter\edef\csname\Upper#2#4\endcsname{\noexpand\Text{\Upper#6#8}}
}

\begin{document}
\title{Clifford modules and invariants of quadratic forms}
\author[ ]{M.~Karoubi}
\address{Universit\'{e} Denis Diderot- Paris 7, UFR de Math\'{e}matiques. Case 7012,
175, rue du Chevaleret. 75205 Paris cedex 13}
\email{max.karoubi@gmail.com}
\thanks{}
\date{8 May 2010}
\maketitle









\pagestyle{myheadings}
\setcounter{section}{-1}%

\section{Introduction}

For any integer $k>0,$ the Bott class $\rho^{k}$ in topological complex
$K$-theory is well known \cite{Bott}, \cite[pg. 259]{Karoubi livre}. If $V$ is
a complex vector bundle on a compact space $X$, $\rho^{k}(V)$ is defined as
the image of $1$ by the composition%
\[
K(X)\overset{\varphi}{\longrightarrow}K(V)\overset{\psi^{k}}{\longrightarrow
}K(V)\overset{\varphi^{-1}}{\longrightarrow}K(X),
\]
where $\varphi$ is Thom's isomorphism in complex $K$-theory and $\psi^{k}$ is
the Adams operation. This characteristic class is natural and satisfies the
following properties which insure its uniqueness (by the splitting principle):

\qquad1) $\rho^{k}(V\oplus W)$ = $\rho^{k}(V).\rho^{k}(W)$

\qquad2) $\rho^{k}(L)$ = $1\oplus L\oplus...\oplus L^{k-1}$ if $L$ is a line bundle.

The Bott class may be extended to the full $K$-theory group if we invert the
number $k$ in the group $K(X)$. It induces a morphism from $K(X)$ to the
multiplicative group $K(X)\left[  1/k\right]  ^{\times}$. The Bott class is
sometimes called "cannibalistic", since both its origin and destination are
$K$-groups.

As pointed out by Serre \cite{Serre}, the definition of the Bott class and its
"square root", introduced in Lemma \ref{Serre}, may be generalized to
$\lambda$-rings, for instance in the theory of group representations or in
equivariant topological $K$-theory.

The purpose of this paper is to give a hermitian analog of the Bott class. We
shall define it on hermitian $K$-theory, with target algebraic $K$-theory. For
instance, let $X=$ \textrm{Spec}$(R),$ where $R$ is a commutative ring with
$k!$ invertible and let $V$ be an algebraic vector bundle on $X$ provided with
a nondegenerate quadratic form\footnote{and also a spinorial structure: see
below.}. We shall associate to $V$ a "hermitian Bott class", designated by
$\rho_{k}(V),$ which takes its values in the same type of multiplicative group
$K(X)\left[  1/k\right]  ^{\times},$ where $K(X)$ is algebraic $K$-theory.

We write $\rho_{k}$ instead of $\rho^{k}$ in order to distinguish the new
class from the old one, although they are closely related (cf. Theorem
\ref{Classical Bott class}).We also note that the "cannibalistic" character of
the new class $\rho_{k}$ is avoided since the source and the target are
different groups. We refer to \cite{Karoubi Annals2} for some basic notions in
hermitian $K$-theory, except that we follow more standard notations, writing
this theory $KQ(X),$ instead of $L(X)$ as in \cite{Karoubi Annals2}.

In order to define the new class $\rho_{k}(V)$, we need a slight enrichment of
hermitian $K$-theory, using "spinorial modules" and not only quadratic ones.
More precisely, a spinorial module is given by a couple $(V,E),$ where $V$ is
a quadratic module and $E$ is a finitely generated projective module, such
that the Clifford algebra $C(V)$ is isomorphic to \textrm{End}$(E).$ The
associated Grothendieck group $K$\textrm{Spin}$(X)$ is related to the
hermitian $K$-group $KQ(X)$ by an exact sequence%
\[
0\longrightarrow\mathrm{Pic}(X)\overset{\theta}{\longrightarrow}%
K\mathrm{Spin}(X)\overset{\varphi}{\longrightarrow}KQ(X)\overset{\gamma
}{\longrightarrow}\mathrm{BW}(X),
\]
where \textrm{BW}$(X)$ denotes the Brauer-Wall group of $X.$ As a set,
\textrm{BW}$(X)$ is isomorphic to the sum of three \'{e}tale cohomology groups
\cite{Wall} \cite[Theorem 3.6]{Caenepell}. There is a twisted group rule on
this direct sum, (compare with \cite{Donovan-Karoubi}). In particular, for the
spectrum of fields, the morphism $\gamma$ is induced by the rank, the
discriminant and the Hasse-Witt invariant \cite{Wall}. From this point of
view, the class $\rho_{k}$ we shall define on $K$\textrm{Spin}$(X)$ may be
considered as a secondary invariant.

The hyperbolic functor $K(X)\longrightarrow KQ(X)$ admits a natural
factorization%
\[
H:K(X)\longrightarrow K\mathrm{Spin}(X)\longrightarrow KQ(X).
\]
The class $\rho_{k}$ is more precisely a homomorphism%
\[
\rho_{k}:K\mathrm{Spin}(X)\longrightarrow K(X)\left[  1/k\right]  ^{\times},
\]
such that we have a factorization with the classical Bott class $\rho^{k}:$%
\[%
\begin{tabular}
[c]{ccc}%
$K(X)$ & $\overset{\rho^{k}}{\longrightarrow}$ & $K(X)\left[  1/k\right]
^{\times}$\\
$H\searrow$ &  & $\nearrow\rho_{k}$\\
& $K$\textrm{Spin}$(X)$ &
\end{tabular}
\ \
\]
An important example is when the bundle of Clifford algebras $C(V)$ has a
trivial class in \textrm{BW}$(X)$. In that case, $C(V)$ is the bundle of
endomorphisms of a $\mathbf{Z}/2$-graded vector bundle $E$ (see the Appendix)
and we can interpret $\rho_{k}$ as defined on a suitable subquotient of
$KQ(X),$ thanks to the exact sequence above. If $k$ is odd, using a result of
Serre \cite{Serre}, we can "correct" the class $\rho_{k}$ into another class
$\overline{\rho}_{k}$ which is defined on the "spinorial Witt group"
\[
W\mathrm{Spin}(X)=\mathrm{Coker}\left[  K(X)\longrightarrow K\mathrm{Spin}%
(X)\right]
\]
and which takes its values in the $2$-torsion of the multiplicative group
$K(X)\left[  1/k\right]  ^{\times}/($\textrm{Pic}$(X))^{(k-1)/2}.$

With the same method, for $n>0,$ we define Bott classes in "higher spinorial
$K$-theory"$:$%
\[
\rho_{k}:K\mathrm{Spin}_{n}(X)\longrightarrow K_{n}(X)\left[  1/k\right]
\]
There is a canonical homomorphism%
\[
K\mathrm{Spin}_{n}(X)\longrightarrow KQ_{n}(X)
\]
which is injective if $n\geq2$ and bijective if $n>2.$ For all $n\geq0$, the
following diagram commutes%
\[%
\begin{tabular}
[c]{ccc}%
$K_{n}(X)$ & $\overset{\rho^{k}}{\longrightarrow}$ & $K_{n}(X)\left[
1/k\right]  $\\
$H\searrow$ &  & $\nearrow\rho_{k}$\\
& $K$\textrm{Spin}$_{n}(X)$ &
\end{tabular}
\ \ .\
\]

In Section $4$, we make the link with Topology, showing that $\rho_{k}$ is
essentially Bott's class defined for spinorial bundles (whereas $\rho^{k}$ is
related to complex vector bundles as we have seen before).

Sections 5 and 6 are devoted to characteristic classes for Azumaya algebras,
especially generalizations of Adams operations.

Finally, in Section 7, we show how to avoid spinorial structures by defining
$\rho_{k}$ on the full hermitian $K$-group $KQ(X)$. The target of $\rho_{k}$
is now an algebraic version of "twisted $K$-theory" \cite{Karoubi tordu}. We
recover the previous hermitian Bott class in a presence of a spinorial structure.

\textbf{Terminology}. It will be implicit in this paper that tensor products
of \textbf{Z}$/2$-graded modules or algebras are graded tensor products.

\textbf{Aknowledgments}. As we shall see many times through the paper, our
methods are greatly inspired by the papers of Bott \cite{Bott}, Atiyah
\cite{Atiyah}, Atiyah, Bott and Shapiro \cite{ABS}, and Bass \cite{Bass
Clifford}. We are indebted to Serre for the Lemma \ref{Serre}, concerning the
"square root" of the classical Bott class. If $k$ is odd, we use this Lemma in
order to define the characteristic class $\overline{\rho}_{k}$ mentioned above
for the Witt group. In Section $7,$ a more refined square root is used.
Finally, we are indebted to Deligne, Knus and Tignol for useful remarks about
operations on Azumaya algebras which are defined briefly in Sections 5 and 6.

Here is a summary of the paper by Sections:

1. Clifford algebras and the spinorial group. Orientation of a quadratic module

2. Operations on Clifford modules

3. Bott classes in hermitian $K$-theory

4. Relation with Topology

5. Oriented Azumaya algebras

6. Adams operations revisited

7. Twisted hermitian Bott classes

Appendix. A remark about the Brauer-Wall group.

\section{Clifford algebras and the spinorial group. Orientation of a quadratic
module}

In this Section, we closely follow a paper of Bass \cite{Bass Clifford}. The
essential prerequisites are recalled here for the reader's convenience and in
order to fix the notations.

Let $R$ be a commutative ring and let $V$ be a finitely generated projective
$R$-module provided with a nondegenerate quadratic form $q.$ We denote by
$C(V,q),$ or simply $C(V),$ the associated Clifford algebra which is naturally
\textbf{Z}$/2$-graded. The canonical map from $V$ to $C(V)$ is an injection
and we shall implicitly identify $V$ with its image.

The Clifford group $\Gamma(V)$ is the subgroup of $C(V)^{\times}$, whose
elements $u$ are homogeneous and satisfy the condition%
\[
uVu^{-1}\subset V.
\]
We define a homomorphism from $\Gamma(V)$ to the orthogonal group
\[
\phi:\Gamma(V)\longrightarrow\mathrm{O}(V)
\]
by the formula
\[
\phi(u)(v)=(-1)^{deg(u)}u.v.u^{-1}.
\]

The group we are interested in is the $0$-degree part of $\Gamma(V),$ i.e.%
\[
\Gamma^{0}(V)=\Gamma(V)\cap C^{0}(V).
\]
We then have an exact sequence proved in \cite[pg. 172]{Bass Clifford}$:$%
\[
1\rightarrow R^{\ast}\rightarrow\Gamma^{0}(V)\rightarrow\mathrm{SO}(V).
\]
The group \textrm{SO}$(V)$ in this sequence is defined as the kernel of the
"determinant map"
\[
\det:\mathrm{O}(V)\rightarrow\mathbf{Z}/2(R),
\]
where \textbf{Z}$/2(R)$ is the set of locally constant functions from
\textrm{Spec}$(R)$ to \textbf{Z}$/2$. This set may be identified with the
Boolean ring of idempotents in the ring $R,$ according to \cite[pg. 159]{Bass
Clifford}. The addition of idempotents is defined as follows%
\[
(e,e^{\prime})\longmapsto e+e^{\prime}-ee^{\prime}.
\]
The determinant map is then a group homomorphism. If \textrm{Spec}$(R)$ is
connected and if $2$ is invertible in $R,$ we recover the usual notion of
determinant which takes its values in the multiplicative group $\pm1.$

We define an antiautomorphism of order $2$ (called an involution through this
paper):
\[
a\longmapsto\overline{a}%
\]
of the Clifford algebra by extension of the identity on $V$ (we change here
the notation of Bass who writes this involution $a\longmapsto{}^{t}a$).

If $a\in\Gamma(V)$, its "spinorial norm" $N(a)$ is given by the formula
\[
N(a)=a\overline{a}.
\]
It is easy to see that $N(a)\in R^{\times}\subset C(V)^{\times}.$ The
spinorial group \textrm{Spin}$(V)$ is then the subgroup of $\Gamma^{0}(V)$
whose elements are of spinorial norm $1.$

We have an exact sequence%
\[
1\rightarrow\mu_{2}(R)\rightarrow\mathrm{Spin}(V)\rightarrow\mathrm{SO}%
(V)\rightarrow\mathrm{Disc}(R).
\]
Here $\mu_{2}(R)$ is the group of $2$-roots of the unity in $R.$ It is reduced
to $\pm1$ if $R$ is an integral domain and if $2$ is invertible in $R$. On the
other hand, \textrm{Disc}$(R)$ is an extension%
\[
1\rightarrow R^{\ast}/R^{\ast}{}^{2}\rightarrow\mathrm{Disc}(R)\rightarrow
\mathrm{Pic}_{2}(R)\rightarrow1,
\]
where \textrm{Pic}$_{2}(R)$ is the $2$-torsion of the Picard group \cite[pg.
176]{Bass Clifford}. The homomorphism%
\[
\mathrm{SN}:\mathrm{SO}(V)\rightarrow\mathrm{Disc}(R),
\]
which is the generalization of the spinorial norm if $R$ is a field, is quite
subtle and is also detailed in \cite{Bass Clifford}.

The map \textrm{SN} stabilizes and defines a homomorphism (where
\textrm{SO}$(R)=\underset{m}{\operatorname{col}\text{im}}$\textrm{SO}%
$(H(R^{m}))$%
\[
\chi:\mathrm{SO}(R)\rightarrow\mathrm{Disc}(R).
\]
The following theorem is proved in \cite[pg. 194]{Bass Clifford}.

\begin{theorem}
The determinant map and the spinorial norm define a homomorphism%
\[
\widetilde{\chi}:\mathrm{O}(R)\rightarrow\mathbf{Z}/2(R)\oplus\mathrm{Disc}%
(R)
\]
which is surjective. It induces a split epimorphism
\[
KQ_{1}(R)\rightarrow\mathbf{Z}/2(R)\oplus\mathrm{Disc}(R).
\]

\end{theorem}

The following corollary is immediate.

\begin{corollary}
We have a central extension%
\[
1\rightarrow\mu_{2}(R)\rightarrow\mathrm{Spin}(R)\rightarrow\mathrm{SO}%
^{0}(R)\rightarrow1,
\]
where \textrm{SO}$^{0}(R)$ is the kernel of the epimorphism $\widetilde{\chi}$
defined above.
\end{corollary}

Let us now assume that $2$ is invertible in $R$ and that the quadratic form
$q$ is defined by a symmetric bilinear form $f$, i.e.%
\[
q(x)=f(x,x).
\]
The symmetric bilinear form associated to $q$ is then $(x,y)\mapsto2f(x,y).$

Let us also assume that $V$ is an $R$-module of constant rank which is even,
say $n=2m.$ In this case, the $n^{th}$ exterior power $\lambda^{n}(V)$ is an
$R$-module of rank $1$ which may be provided with the quadratic form
associated to $q.$ We say that $V$ is orientable (in the quadratic sense) if
$\lambda^{n}(V)$ is isomorphic to $R$ with the standard quadratic form
$\theta:x\longmapsto x^{2}$ (up to a scaling factor which is a square). We say
that $V$ is oriented if we fix an isometry between $\lambda^{n}(V)$ and
$(R,\theta).$ If $V$ is free with a given basis, this is equivalent to saying
that the symmetric matrix associated to $f$ is of determinant $1.$

\begin{remark}
One may use the orientation on $V$ to define on $C^{0}(V)$ a symmetric
bilinear form%
\[
\Phi^{0}:C^{0}(V)\times C^{0}(V)\rightarrow C^{0}(V)\overset{\sigma
}{\rightarrow}\lambda^{n}(V)\cong R.
\]
The last map $\sigma$ is defined by the canonical filtration of the Clifford
algebra, the associated graded algebra being the exterior algebra. In the same
way, we define an antisymmetric form by taking the composition%
\[
\Phi^{1}:C^{1}(V)\times C^{1}(V)\rightarrow C^{0}(V)\overset{\sigma
}{\rightarrow}\lambda^{n}(V)\cong R.
\]

The following theorem is not really needed for our purposes but is worth recording.
\end{remark}

\begin{theorem}
The previous bilinear forms $\Phi^{0}$ and $\Phi^{1}$ are non degenerate, i.e.
induce isomorphisms between $C(V)$ and its dual as an $R$-module.
\end{theorem}

\begin{proof}
We can check this Theorem by localizing at any maximal ideal $(m)$ (see for
instance \cite[pg. $49$]{Atiyah-MacDonald}). In this case, there exists an
orthogonal basis $(e_{1},...,e_{n})$ of $V_{(m)}.$ Since $V$ is oriented, we
may choose this basis such that the product $q(e_{1})...q(e_{n})$ is equal to
$1.$ It is also well known that the various products%
\[
e_{I}=e_{i_{1}...}e_{i_{r}}%
\]
form a basis of the free $R_{(m)}$-module $C(V_{(m)})$. Here the multiindex
$I=(i_{1},...,i_{r})$ is chosen such that $i_{1}<i_{2}<...<i_{r}.$ By a direct
computation we have%
\[
\Phi(e_{I},e_{J})=\pm1
\]
if $I\cup J=\left\{  1,...,n\right\}  $ and $0$ otherwise, for $\Phi=\Phi_{0}$
or $\Phi_{1}$. Therefore, these bilinear forms are non degenerate. Moreover,
they are hyperbolic at each localization.
\end{proof}

\begin{remark}
Since $V$ is oriented, the group $\mathrm{SO}(V)$ acts naturally on $C(V)$ and
we get two natural representations of this group in the orthogonal and
symplectic groups associated to the previous bilinear forms $\Phi^{0}$ and
$\Phi^{1}.$
\end{remark}

Let us now consider the submodule $N$ of $C(V)$ whose elements $u$ satisfy the
identity $u.v=-v.u$ for any element $v$ in $V\subset C(V).$ The canonical
surjection $V\longrightarrow\lambda^{n}(V)$ induces a homomorphism
\[
\tau:N\longrightarrow\lambda^{n}(V).
\]

\begin{proposition}
The homomorphism $\tau$ is an isomorphism between $N$ and $\lambda^{n}(V).$
Moreover, $N$ is included in $C^{0}(V).$
\end{proposition}

\begin{proof}
We again localize with respect to all maximal ideals $(m)$ of $R$ and consider
an orthogonal basis $\left\{  e_{i}\right\}  $ of $V_{(m)}$ as above. Then we
see that the product $e_{1}...e_{n}$ generates $N$ and we get the required
isomorphism between $N_{(m)}$ and $\lambda^{n}(V)_{(m)}.$
\end{proof}

\begin{remark}
\label{Definition of u}If we assume that $V$ is oriented and of even rank, the
previous proposition provides us with a canonical element $u$ in $C^{0}(V)$
which anticommutes with all elements $v$ in $V$, such that $u^{2}=1$.
Moreover, $u.\overline{u}=1$ and therefore $u$ belongs to the spinorial group
$\mathrm{Spin}(V).$
\end{remark}

An important example is the case when the Clifford algebra $C(V)$ has a
trivial class in the Brauer-Wall group of $R,$ denoted by \textrm{BW}%
$(R)\footnote{We shall also use the notation $\mathrm{BW}(X)$ if
$X=\mathrm{Spec}(R),$ as we wrote before.}.$ In other words, $C(V)$ is
isomorphic to the algebra $\textrm{End}(E)$ of a graded vector space
$E=E_0\oplus E_1$ where $E_0$ and $E_1$ are not reduced to $0$ (see the
Appendix$).$ The only possible choices for $u$ are then one of the two
following matrices%
\[
\left[
\begin{array}
[c]{cc}%
1 & 0\\
0 & -1
\end{array}
\right]  \text{or}\left[
\begin{array}
[c]{cc}%
-1 & 0\\
0 & 1
\end{array}
\right]
\]

We always choose $E$ such that $u$ is of the first type and, by a topological
analogy, we shall say that $V$ is "spinorial". For instance, let $R$ be the
ring of real continuous functions on a compact space $X$ and let $V$ be a real
vector bundle provided with a positive definite quadratic form. The triviality
of the Clifford bundle $C(V)$ in \textrm{BW}$(X)$ is then equivalent to the
following properties: the rank of $V$ is a multiple of $8$ and the two first
Stiefel-Whitney classes $w_{1}(V)$ and $w_{2}(V)$ are trivial (see
\cite{Donovan-Karoubi}).

\begin{remark}
Strictly speaking, in the topological situation, the classical spinoriality
property does not imply that the rank of $V$ is a multiple of $8.$ We put this
extra condition in order to ensure the trivialization of $C(V)$ in the
Brauer-Wall group of $R$.
\end{remark}

\section{Operations on Clifford modules}

As it is well known, at least for fields, the standard non trivial invariants
of quadratic forms $(V,q)$ are the discriminant and the Hasse-Witt invariant.
They are encoded in the class of the Clifford algebra $C(V)=C(V,q)$ in the
Brauer-Wall group of $R,$ which we call \textrm{BW}$(R)$, as in the previous
Section. For any commutative ring $R$, this group \textrm{BW}$(R)$ has been
computed by Wall and Caenepeel \cite{Wall}\cite{Caenepell}. As a set, it is
the sum of the first three \'{e}tale cohomology groups of $X=$ \textrm{Spec}%
$(R)$ but with a twisted group rule (compare with \cite{Donovan-Karoubi}). We
view this class of $C(V)$ in \textrm{BW}$(R)$ as a "primary" invariant. In
order to define "secondary" invariants, we may proceed as usual by assuming
first that this class is trivial. Therefore, we have an isomorphism%
\[
C(V)\cong\mathrm{End}(E),
\]
where $E$ is a \textbf{Z}$/2$-graded $R$-module which is projective and
finitely generated. We always choose $E$ such that the associated element $u$
defined in the previous section is the matrix%
\[
\left[
\begin{array}
[c]{cc}%
1 & 0\\
0 & -1
\end{array}
\right]  .
\]
However, $E$ is not uniquely defined by these conditions. If
\[
\mathrm{End}(E)\cong\mathrm{End}(E^{\prime}),
\]
we have $E^{\prime}\cong E\otimes L,$ where $L$ is a module of rank $1,$
concentrated in degree $0$ according to our choice of $u$ (this is a simple
consequence of Morita equivalence).

As in the introduction, we may formalize the previous considerations better
thanks to the following definition. A "spinorial module" is a couple $(V,E),$
where $E$ is a finitely generated projective module and $V=(V,q)$ is a
quadratic oriented module, such that $C(V)$ is isomorphic to \textrm{End}$(E)$
with the choice of $u$ above. We define the "sum" $(V,E)+(V^{\prime}%
,E^{\prime})$ as $(V\oplus V^{\prime},E\otimes E^{\prime})$ and the group
$K$\textrm{Spin}$(R)$ by the usual Grothendieck construction.

\begin{proposition}
We have an exact sequence%
\[
0\longrightarrow\mathrm{Pic}(R)\overset{\theta}{\longrightarrow}%
K\mathrm{Spin}(R)\overset{\varphi}{\longrightarrow}KQ(R)\overset{\gamma
}{\longrightarrow}\mathrm{BW}(R),
\]
where the homomorphisms $\gamma,\varphi$ and $\theta$ are defined below.
\end{proposition}

\begin{proof}
The map $\gamma$ was defined previously: it associates to the quadratic module
$(V,q)$ the class of the Clifford algebra $C(V)=C(V,q)$ in $BW(R).$ We note
that $\gamma$ is not necessarily surjective, even on the $2$-torsion part: see
\cite[pg. $11$]{Donovan-Karoubi} for counterexamples. The map $\varphi$ sends
a couple $(V,E)$ to the class of the quadratic module $V.$ Finally, $\theta$
associates to a module $L$ of rank one the difference\footnote{Note that we
can replace $R$ by $R^{n}$ in this formula.} $(H(R),\Lambda(R)\otimes
L)-(H(R),\Lambda(R)),$ where $H$ is the hyperbolic functor and $\Lambda$ the
exterior algebra functor, viewed as a module functor. This map $\theta$ is a
homomorphism since the image of $L\otimes L^{\prime}$ may be written as
follows%
\[
(H(R),\Lambda(R)\otimes L\otimes L^{\prime})-(H(R),\Lambda(R)\otimes L)
\]%
\[
+(H(R),\Lambda(R)\otimes L)-(H(R),\Lambda(R)).
\]
This image is also%

\[
(H(R),\Lambda(R)\otimes L^{\prime})-(H(R),\Lambda(R))+(H(R),\Lambda(R)\otimes
L)-(H(R),\Lambda(R))
\]
which is $\theta(L)+\theta(L^{\prime}).$ In order to complete the proof, it
remains to show that the induced map%
\[
\sigma:\mathrm{Pic}(R)\longrightarrow\mathrm{Ker}(\varphi)
\]
is an isomorphism.

1) The map $\sigma$ is surjective. Any element of \textrm{Ker}$(\varphi)$ may
be written $(V,E)-(V,E^{\prime}).$ Therefore, we have $E\cong E^{\prime
}\otimes L,$ where $L$ is of rank $1.$ If we add to this element
$(H(R),\Lambda(R)\otimes L^{-1})-(H(R),\Lambda(R)),$ which belongs to
$\operatorname{Im}(\varphi),$ we find $0.$

2) The map $\sigma$ is injective. We define a map backwards
\[
\sigma^{\prime}:\mathrm{Ker}(\varphi)\longrightarrow\mathrm{Pic}(R),
\]
by sending the difference $(V,E)-(V,E^{\prime})$ to the unique $L$ such that
$E\cong E^{\prime}\otimes L.$ It is clear that $\sigma^{\prime}\cdot
\sigma=Id,$ which proves the injectivity of $\sigma.$
\end{proof}

Before going any further, we need a convenient definition, due to Atiyah, Bott
and Shapiro \cite{ABS}, of the graded Grothendieck group $GrK(A)$ of a
\textbf{Z}$/2$-graded algebra $A.$ It is defined as the cokernel of the
restriction map%
\[
K(A\widehat{\otimes}C^{0,2})\longrightarrow K(A\widehat{\otimes}C^{0,1}),
\]
where $C^{0,r}$ is in general the Clifford algebra of $R^{r}$ with the
standard quadratic form $\sum_{i=1}^{r}(x_{i})^{2}.$ We note that if $A$ is
concentrated in degree $0,$ we recover the usual definition of the
Grothendieck group $K(A),$ under the assumption that $2$ is invertible in $A,$
which we assume from now on. This follows from the fact that a \textbf{Z}%
$/2$-graded structure on a module $M$ is equivalent to an involution on $M.$

Another important example is $A=C(V,q),$ where $V$ is oriented and of even
rank. In order to compute the graded Grothendieck group of $A$, we use the
element $u$ introduced in \ref{Definition of u} to define a natural
isomorphism%
\[
A\widehat{\otimes}C^{0,r}\longrightarrow A\otimes C^{0,r}.
\]
It is induced by the map
\[
(v,t)\mapsto v\otimes1+u\otimes t,
\]
where $v\in V\subset C(V,q)$ and $t\in R^{r}\subset C^{0,r}$. If $E$ is a
\textbf{Z}$/2$-graded $R$-module, the same argument may be applied to $A=$
\textrm{End}$(E).$ The graded Grothendieck group again coincides with the
usual one. Since we consider only these examples in our paper, we simply write
$K(A)$ instead of $GrK(A)$ from now on.

The graded algebras we are interested in are the Clifford algebras
$\Lambda_{k}=$ $C(V,kq),$ where $k>0$ is an invertible integer in $R$. The
interest of this family of algebras is the following. Let $M$ be a
\textbf{Z}$/2$-graded module over $\Lambda_{1}.$ Then its $k^{th}$-power
$M^{\widehat{\otimes}k}$ is a graded module over the crossed product algebra
$S_{k}\ltimes C(V)^{\widehat{\otimes}k})\cong S_{k}\ltimes C(V^{k}),$ where
$S_{k}$ is the symmetric group on $k$ letters. One has to remark that the
action of the symmetric group $S_{k}$ on $M^{\widehat{\otimes}k}$ takes into
account the grading as in \cite[pg. $176$]{Atiyah}: the transposition $(i,j)$
acts on a decomposable homogeneous tensor
\[
m_{1}\otimes...\otimes m_{i}\otimes...\otimes m_{j}\otimes...\otimes m_{k}%
\]
as the permutation of $m_{i}$ and $m_{j},$ up to the sign $(-1)^{\deg
(m_{i})\deg(m_{j})}.$

Let us now consider the diagonal $V\longrightarrow V^{k}.$ It is an isometry
if we provide $V$ with the quadratic form $kq.$ Therefore, we have a well
defined map%
\[
\Lambda_{k}\longrightarrow C(V^{k})
\]
which is equivariant with respect to the action of the symmetric group
$S_{k}.$ It follows that the correspondence%
\[
M\longmapsto M^{\widehat{\otimes}k}%
\]
induces (by restriction of the scalars) a "power map"%
\[
P:K(\Lambda_{1})\longrightarrow K_{S_{k}}(\Lambda_{k}),
\]
where $K_{S_{k}}$ denotes equivariant $K$-theory, the group $S_{k}$ acting
trivially on $\Lambda_{k}$.

Let us give more details about this definition. First, we notice that $V^{k}$
splits as the direct sum of $(V,kq)$ and its orthogonal module $W$. This
implies that $C(V)^{\widehat{\otimes}k}\cong C(V^{k})\cong C(W)\widehat
{\otimes}\Lambda_{k}$ is a $\Lambda_{k}$-module which is finitely generated
and projective. Therefore, the "restriction of scalars" functor from the
category of finitely generated projective modules over $C(V^{k})$ to the
analogous category of modules over $\Lambda_{k}$ is well defined. Secondly, we
have to show that the map $P,$ which is a priori defined in terms of modules,
can be extended to a map between graded Grothendieck groups. This may be shown
by using a trick due to Atiyah which is detailed in \cite[pg. $175$]%
{Atiyah}\footnote{More precisely, Atiyah is considering complexes in his
argument but the same idea may be applied to \textbf{Z}$/2$-graded modules.}.
Finally, we notice that $P$ is a \textbf{set} map, not a group homomorphism.

In order to define $K$-theory operations in this setting, we may proceed in at
least two ways. First, following Grothendieck, we consider a \textbf{Z}%
$/2$-graded module $M$ and its $k^{th}$-exterior power in the graded sense.
The specific map $K_{S_{k}}(\Lambda_{k})\longrightarrow K(\Lambda_{k})$ which
defines the $k^{th}$-exterior power is the following: we take the quotient of
$M^{\widehat{\otimes}k}$ by the relations identifying to $0$ all the elements
of type%
\[
m-\varepsilon(\sigma)m^{\sigma}.
\]
In this formula, $m$ is an element of $M^{\widehat{\otimes}k},m^{\sigma}$ its
image under the action of the element $\sigma$ in the symmetric group, with
signature $\varepsilon(\sigma)$. The composition%
\[
K(\Lambda_{1})\longrightarrow K_{S_{k}}(\Lambda_{k})\longrightarrow
K(\Lambda_{k})
\]
defines the analog of Grothendieck's $\lambda$-operations:%
\[
\lambda^{k}:K(\Lambda_{1})\longrightarrow K(\Lambda_{k}),
\]
as detailed in \cite[pg. $252$]{Karoubi livre} for instance.

\begin{remark}
If $M$ is a graded module concentrated in degree $0$ (resp. $1)$ $\lambda
^{k}(M)$ is the usual exterior power (resp. symmetric power) with an extra
$\Lambda_{k}$-module structure.
\end{remark}

The diagonal map from $V$ into $V\times V$ enables us to define a
"cup-product": it is induced by the tensor product of modules with Clifford
actions:%
\[
K(\Lambda_{k})\times K(\Lambda_{l})\longrightarrow K(\Lambda_{k+l})
\]
The following theorem is a consequence of the classical property of the usual
exterior (graded) powers, extended to this slightly more general situation.

\begin{theorem}
Let $M$ and $N$ be two $\Lambda_{1}$-modules. Then one has natural
isomorphisms of $\Lambda_{r}$-modules%
\[
\lambda^{r}(M\oplus N)\cong\underset{k+l=r}{%
{\displaystyle\sum}
}\lambda^{k}(M)\otimes\lambda^{l}(N).
\]

\end{theorem}

\begin{proof}
It is more convenient to consider the direct sum of all the $\lambda^{k}(M)$,
which we view as the \textbf{Z}-graded exterior algebra $\Lambda(M).$ Since
the natural algebra isomorphism%
\[
\Lambda(M)\otimes\Lambda(N)\longrightarrow\Lambda(M\oplus N)
\]
is compatible with the Clifford structures, the theorem is proved.
\end{proof}

From these $\lambda$-operations, it is classical to associate "Adams
operations" $\Psi^{k}.$ For any element $x$ of $K(\Lambda_{1}),$ we define
$\Psi^{k}(x)\in K(\Lambda_{k})$ by the formula%
\[
\Psi^{k}(x)=Q_{k}(\lambda^{1}(x),...,\lambda^{k}(x)),
\]
where $Q_{k}$ is the Newton polynomial (cf. \cite[pg. 253]{Karoubi livre} for
instance). The following theorem is a formal consequence of the previous one.

\begin{theorem}
Let $x$ and $y$ be two elements of $K(\Lambda_{1}).$Then one has the identity%
\[
\Psi^{k}(x+y)=\Psi^{k}(x)+\Psi^{k}(y)
\]
in the group $K(\Lambda_{k}).$
\end{theorem}

\begin{proof}
Following Adams \cite[pg. $257$]{Karoubi livre}, we note that the series%
\[
\Psi_{-t}(x)=\underset{k=1}{\overset{\infty}{%
{\textstyle\sum}
}}(-1)^{k}t^{k}\Psi^{k}(x)
\]
is the logarithm differential of $\lambda_{t}(x)$ multiplied by $-t,$ i.e.%
\[
-t\frac{\lambda_{t}^{^{\prime}}(x)}{\lambda_{t}(x)}.
\]
This can be checked by a formal "splitting principle" as in \cite{Karoubi
livre} for instance. The additivity of the Adams operation follows from the
fact that the logarithm differential of a product is the sum of the logarithm
differentials of each factor.
\end{proof}

\bigskip\medskip

Another important and less obvious property of the Adams operations is the following.

\begin{theorem}
\label{multiplicativity Adams}Let us assume that $k!$ is invertible in $R$ and
let $x$ and $y$ be two elements of $K(\Lambda_{1}).$ Then one has the
following identity in the group $K(\Lambda_{2k}).$%
\[
\Psi^{k}(x\cdot y)=\Psi^{k}(x)\cdot\Psi^{k}(y).
\]

\end{theorem}

\begin{proof}
In order to prove this theorem, we use the second description of the
operations $\lambda^{k}$ and $\Psi^{k}$ due to Atiyah \cite[\S \ 2]{Atiyah},
which we transpose in our situation. In order to define operations in
$K$-theory, Atiyah considers the following composition (where $R(S_{k})$
denotes the integral representation group ring of the symmetric group $S_{k}%
$):%
\[
K(\Lambda_{1})\overset{P}{\longrightarrow}K_{S_{k}}(\Lambda_{k})\overset
{\cong}{\longrightarrow}K(\Lambda_{k})\otimes R(S_{k})\overset{\chi
}{\longrightarrow}K(\Lambda_{k}).
\]
In this sequence, $P$ is the $k^{th}$-power map introduced before. The second
map is defined by using our hypothesis that $k!$ is invertible in $R.$ More
precisely, any $S_{k}$-module is semi-simple and is therefore the direct sum
of its isotopy summands: if $\pi$ runs through all the (integral) irreducible
representations of the symmetric group $S_{k},$ the natural map%
\[
\oplus\mathrm{Hom}(\pi,T)\otimes\pi\longrightarrow T
\]
is an isomorphism (note that $\pi$ is of degree $0)$. Therefore, by linearity,
the equivariant $K$-theory $K_{S_{k}}(C(V,kq))=$ $K_{S_{k}}(\Lambda_{k})$ may
be written as $K(\Lambda_{k})\otimes R(S_{k}).$ Finally, the map $\chi$ is
defined once a homomorphism%
\[
\chi_{k}:R(S_{k})\longrightarrow\mathbf{Z}%
\]
is given. For instance, the Grothendieck operation $\lambda^{k}(M)$ is
obtained through the specific homomorphism $\chi_{k}$ equal to $0$ for all the
irreducible representations of $S_{k},$ except the sign representation
$\varepsilon$, where $\chi_{k}(\varepsilon)=$1.

\smallskip

Moreover, we can define the product of two operations associated to $\chi_{k}$
and $\chi_{l}$ using the ring structure on the direct sum $\oplus$%
\textrm{Hom}$(R(S_{r}),\mathbf{Z}),$ as detailed in \cite[pg. 169]{Atiyah}.
This structure is induced by the pairing%
\begin{align*}
\mathrm{Hom}(R(S_{k}),\mathbf{Z})\times\mathrm{Hom}(R(S_{l}),\mathbf{Z})  &
\longrightarrow\mathrm{Hom}(R(S_{k}\times S_{l}),\mathbf{Z})\\
&  \longrightarrow\mathrm{Hom}(R(S_{k+l}),\mathbf{Z}).
\end{align*}
In particular, as proved formally by Atiyah \cite[pg. 179]{Atiyah}, the Adams
operation $\Psi^{k}$ is induced to the homomorphism%
\[
\Psi:R(S_{k})\longrightarrow\mathbf{Z}%
\]
associating to a class of representations $\rho$ the trace of $\rho(c_{k}),$
where $c_{k}$ is the cycle $(1,2,...,k).$ With this interpretation, the
multiplicativity of the Adams operation is obvious.
\end{proof}

\begin{remark}
We conjecture that the previous theorem is true without the hypothesis that
$k!$ is invertible in $R.$ If we assume that $2k$ is invertible in $R,$ and
that $R$ contains the $k^{th}$-roots of the unity, we propose another closely
related operation $\overline{\Psi}^{k}$ in Section $6$. We conjecture that
$\Psi^{k}=\overline{\Psi}^{k}$. This is at least true if $k!$ is invertible in
$R.$
\end{remark}

\section{Bott classes in hermitian $K$-theory}

\bigskip

Let us assume that $k!$ is invertible in $R$. We consider a spinorial module
$(V,E),$ as in the previous Section. The following maps are detailed below:%
\begin{align*}
\theta_{k}  &  :K(R)\overset{\alpha}{\underset{\cong}{\longrightarrow}%
}K(C(V,q))\overset{P}{\rightarrow}K_{S_{k}}(C(V,kq))\underset{\cong%
}{\longrightarrow}K(C(V,kq))\otimes R(S_{k})\\
&  \overset{\Psi^{\prime}}{\longrightarrow}K(C(V,kq))\overset{(\alpha
_{q})^{-1}}{\underset{\cong}{\longrightarrow}}K(R).
\end{align*}
The morphism $\alpha$ is the Morita isomorphism between $K(R)$ and
$K(C(V,q))\cong K($\textrm{End}$(E))$ and $P$ is the $k^{th}$-power map
defined in the previous Section. The morphism $\Psi^{\prime}$ is induced by
$\Psi:R(S_{k})\longrightarrow\mathbf{Z}$ also defined there. Finally, for the
definition of $\alpha_{q},$ we remark that the isomorphism between $C(V,q)$
and \textrm{End}$(E)$ implies the existence of an $R$-module map%
\[
f:V\longrightarrow\mathrm{End}(E_{0}\oplus E_{1})
\]
such that%
\[
f(v)=\left[
\begin{array}
[c]{cc}%
0 & \sigma(v)\\
\tau(v) & 0
\end{array}
\right]
\]
with $\sigma(v)\tau(v)=\tau(v)\sigma(v)=q(v).1.$ We now define a "$k$-twisted
map"
\[
f_{k}:V\longrightarrow\mathrm{End}(E_{0}\oplus E_{1})
\]
by the formula%
\[
f_{k}(v)=\left[
\begin{array}
[c]{cc}%
0 & k\sigma(v)\\
\tau(v) & 0
\end{array}
\right]  .
\]
Since $(f_{k}(v))^{2}=kq(v)$, $f_{k}$ induces a homomorphim between $C(V,kq)$
and \textrm{End}$(E_{0}\oplus E_{1})$ which is clearly an isomorphism, as we
can see by localizing at all maximal ideals. The map $\alpha_{q}$ is then
induced by the same type of Morita isomorphism we used to define $\alpha.$

$.$

\begin{theorem}
Let $(V,E)$ be a spinorial module and let $M$ be an $R$-module. Then the image
of $M$ by the previous composition $\theta_{k}$ is defined by the following
formula%
\[
\theta_{\kappa}(M)=\rho_{k}(V,E).\Psi^{k}(M).
\]
Therefore, $\theta_{k}$ is determined by $\theta_{\kappa}(1)=\rho_{k}(V,E),$
which we shall simply write $\rho_{k}(V,q)$ or $\rho_{k}(V)$ if the quadratic
form $q$ and the module $E$ are implicit. We call $\rho_{k}(V)$ the "hermitian
Bott class" of $V.$ Moreover, we have the multiplicativity formula%
\[
\rho_{k}(V\oplus W)=\rho_{k}(V)\cdot\rho_{k}(W)
\]
in the Grothendieck group $K(R).$
\end{theorem}

\begin{proof}
The first formula follows from the multiplicativity of the Adams operation
proved in Theorem \ref{multiplicativity Adams}. The second one follows from
the same multiplicativity property and the well-known isomorphism%
\[
C(V\oplus W)\cong C(V)\otimes C(W)
\]
(graded tensor product as always, according to our conventions).
\end{proof}

\begin{theorem}
\label{(V,q) and (V,-q)}Let $(V,-q)$ be the module $V$ provided with the
opposite quadratic form. Then we have the identity%
\[
\rho_{k}(V,q)=\rho_{k}(V,-q).
\]

\end{theorem}

\begin{proof}
According to our hypothesis, the Clifford algebra $C(V)$ is oriented, since it
is isomorphic to $\mathrm{End}(E).$ Therefore, we can use the element $u$
defined in \ref{Definition of u} to show that $C(V,q)$ is isomorphic to
$C(V,-q)$ (more generally, $C(V,q)$ is isomorphic to $C(V,kq)$ if $k$ is
invertible). More explicitly, we keep the same $E$ as the module of spinors,
so that $C(V,q)\cong C(V,-q)\cong\mathrm{End}(E).$ We now write the
commutative diagram%
\[%
\begin{array}
[c]{ccccccc}%
K(R) & \rightarrow & K(C(V,q)) & \overset{\Psi^{k}}{\longrightarrow} &
K(C(V,kq)) & \rightarrow & K(R)\\
\downarrow Id &  & \downarrow\cong &  & \downarrow\cong &  & \downarrow Id\\
K(R) & \rightarrow & K(C(V,-q)) & \overset{\Psi^{k}}{\longrightarrow} &
K(C(V,-kq)) & \rightarrow & K(R)
\end{array}
.
\]

\end{proof}

\begin{remark}
The isomorphism between the Clifford algebras $C(V,q)$ and $C(V,-q)$ is
defined by using the element $u$ of degree $0$ and of square $1$ in $C(V,q)$
which anticommutes with all the elements of $V.$ It is easy to see that the
$k$-tensor product $u_{k}=$ $u\otimes...\otimes u$ satisfies the same
properties for the Clifford algebra $C(V^{k},q\oplus...\oplus q).$ Therefore,
we have an analogous commutative diagram with the power map $P$ instead of the
Adams operation $\Psi^{k}$:%
\[%
\begin{array}
[c]{ccc}%
K(C(V,q)) & \overset{P}{\longrightarrow} & K(C(V,kq))\otimes R(S_{k})\\
\downarrow\cong &  & \downarrow\cong\\
K(C(V,-q)) & \overset{P}{\longrightarrow} & K(C(V,-kq))\otimes R(S_{k})
\end{array}
\]

\end{remark}

\begin{theorem}
\label{Classical Bott class}Let $(V,q)$ be the hyperbolic module $H(P)$ and
$E=\Lambda(P)$ be the associated module of spinors. Then $\rho_{k}(V,E)$ is
the classical Bott class $\rho^{k}(P)$ of the $R$-module $P.$
\end{theorem}

\begin{proof}
According to \cite[pg. 166]{Bass Clifford}, the Clifford algebra $C(V)$ is
isomorphic to \textrm{End}$(\Lambda P)$ as a \textbf{Z}$/2$-graded algebra,
which gives a meaning to our definition. The class $\rho_{k}(V,q)$ may be
identified with the "formal quotient" $\Psi^{k}(\Lambda P)/\Lambda P$ which
satisfies the algebraic splitting principle. Therefore, in order to prove the
theorem, it is enough to consider the case when $P=L$ is of rank one. We have
then $\Lambda L=1-L,$ $\Psi^{k}(\Lambda L)=1-L^{k}$ and therefore, $\Psi
^{k}(\Lambda L)/\Lambda L=1+L+...+L^{k-1}$.
\end{proof}

\medskip

In order to extend the definition of the hermitian Bott class to "spinorial
$K$-theory", we remark that any element $x$ of $K$\textrm{Spin}$(R)$ may be
written as
\[
x=V-H(R^{m}),
\]
where $V$ is a quadratic module (the module of spinors $E$ being implicit).
Moreover, $V-H(R^{m})=V^{\prime}-H(R^{m^{\prime}})$ iff we have an isomorphism%
\[
V\oplus H(R^{m^{\prime}})\oplus H(R^{s})\cong V^{\prime}\oplus H(R^{m})\oplus
H(R^{s})
\]
for some $s.$ Therefore, if we invert $k$ in the Grothendieck group $K(R),$
the following definition%
\[
\rho_{k}(x)=\rho_{k}(V-H(R^{m}))=\rho_{k}(V)/k^{m}%
\]
does not depend of the choice of $V$ and $m.$

The previous definitions are not completely satisfactory if we are interested
in characteristic classes for the "spinorial Witt group" of $R,$ denoted by
$W$\textrm{Spin}$(R)$ and defined as the cokernel of the hyperbolic map%
\[
K(R)\longrightarrow K\mathrm{Spin}(R).
\]
One way to deal with this problem is to consider the underlying module $V_{0}$
of $(V,E).$ According to our hypothesis, $V_{0}$ is a module of even rank,
oriented and isomorphic to its dual. The following lemma is a particular case
of a theorem due to Serre \cite{Serre}. For completeness' sake, we summarize
Serre's formula in this special case.

\begin{lemma}
\label{Serre}Let us assume that $k$ is odd. With the previous hypothesis, the
classical Bott class $\rho^{k}(V_{0})$ is canonically a square in $K(R)$.
\end{lemma}

\begin{proof}
Let $\Omega_{k}$ be the ring of integers in the $k$-cyclotomic extension of
$\mathbb{Q}$ and let $z$ be a primitive $k^{th}$-root of the unity. In the
computations below, we always embed an abelian group $G$ in $G\otimes
_{\mathbf{Z}}\Omega_{k}$. Let us now write%
\[
G_{V_{0}}(t)=1+t\lambda^{1}(V_{0})+...+t^{n}\lambda^{n}(V_{0}).
\]
From the algebraic splitting principle, it follows that%
\[
\rho^{k}(V_{0})=\overset{k-1}{\underset{r=1}{%
{\displaystyle\prod}
}}G_{V_{0}}(-z^{r}).
\]
The identity $\lambda^{j}(V_{0})=\lambda^{n-j}(V_{0})$ implies that $G_{V_{0}%
}(t)=t^{n}G_{V_{0}}(1/t).$ We then deduce from \cite{Serre} that $\rho
^{k}(V_{0})$ has a square root\footnote{We have inserted a normalization sign
$(-1)^{n(k-1)/4}$ before Serre's formula \cite{Serre} for a reason explained
in the computation below.} which we may choose to be%
\[
\sqrt{\rho^{k}(V_{0})}=(-1)^{n(k-1)/4}\overset{(k-1)/2}{\underset{r=1}{%
{\displaystyle\prod}
}}G_{V_{0}}(-z^{r})\cdot z^{-nr/2}.
\]
This square root is invariant under the action of the Galois group
$($\textbf{Z}$/k)^{\ast}$ of the cyclotomic extension which is generated by
the transformations $z\longmapsto z^{j},$ where $j\in($\textbf{Z}$/k)^{\ast}.$
Therefore, it belongs to $K(R),$ as a subgroup of $K(R)\otimes_{\mathbf{Z}%
}\Omega_{k}.$
\end{proof}

The previous lemma enables us to "correct" the hermitian Bott class in the
following way. We put%
\[
\overline{\rho}_{k}(V)=\rho_{k}(V)(\sqrt{\rho^{k}(V_{0})})^{-1}.
\]
If $V$ is a hyperbolic module $H(W)=W\oplus W^{\ast},$ we have $\rho
_{k}(V)=\rho^{k}(W).$ On the other hand, we have $\lambda_{t}(W\oplus W^{\ast
})=J(t)\cdot t^{n/2}\cdot J(1/t)\cdot\sigma,$ where%
\[
J(t)=\lambda_{t}(W)=1+t\lambda^{1}(W)+...+t^{n/2}\lambda^{n/2}(W)
\]
and $\sigma=\lambda^{n/2}(W^{\ast}).$ Therefore,
\[
\sqrt{\rho^{k}((W\oplus W^{\ast})}=(-1)^{n(k-1)/4}\overset{(k-1)/2}%
{\underset{r=1}{%
{\displaystyle\prod}
}}\sigma\cdot J(-z^{r})\cdot(-z^{r})^{n/2}.J(-1/z^{r})\cdot(z^{-nr/2})
\]%
\[
=\sigma^{(k-1)/2}\cdot%
{\textstyle\prod\limits_{r=1}^{k-1}}
J(-z^{r})=\sigma^{(k-1)/2}\cdot\rho^{k}(W)=\sigma^{(k-1)/2}\cdot\rho
_{k}(H(W)).
\]

\medskip

From this computation, it follows that $\overline{\rho}_{k}(V)$ is a $\left[
k-1)/2\right]  ^{th}$-power of an element of the Picard group of $R$ if $V$ is
hyperbolic. Moreover,
\[
\overline{\rho}_{k}(V)^{2}=(\rho_{k}(V))^{2}(\rho^{k}(V_{0}))^{-1}=\rho
_{k}(V,q))\rho_{k}(V,-q))(\rho^{k}(V_{0}))^{-1}%
\]%
\[
=\rho_{k}(H(V_{0}))(\rho^{k}(V_{0}))^{-1}=\rho^{k}(V_{0})(\rho^{k}%
(V_{0}))^{-1}=1.
\]

\bigskip Summarizing this discussion, we have proved the following theorem:

\begin{theorem}
\label{Corrected Bott class}Let $k>0$ be an odd number$.$ Then the corrected
hermitian Bott class $\overline{\rho}_{k}(V)=\rho_{k}(V)(\sqrt{\rho^{k}%
(V_{0})})^{-1}$ induces a homomorphism also called $\overline{\rho}_{k}:$
\[
\overline{\rho}_{k}:W\mathrm{Spin}(R)\longrightarrow(K(R)\left[  1/k\right]
)^{\times}/\mathrm{Pic}(R)^{(k-1)/2},
\]
where the right hand side is viewed as a multiplicative group. Moreover, the
image of $\overline{\rho}_{k}$ lies in the $2$-torsion of this group.
\end{theorem}

\begin{remark}
The case $k$ even does not fit with this strategy. However, we shall see in
the next Section that $\rho_{2}$ is not trivial in general on the Witt group.
\end{remark}

\begin{remark}
We have chosen the Adams operation to define the hermitian Bott class. We
could as well consider any operation induced by a homomorphism%
\[
R(S_{k})\longrightarrow\mathbf{Z}.
\]
The only reason for our choice is the very pleasant properties of the Adams
operations with respect to direct sums and tensor products of Clifford modules.
\end{remark}

We would like to extend the previous considerations to higher hermitian
$K$-theory. More precisely, the orthogonal group we are considering to define
this $K$-theory is the group $O_{m,m}(R)$ which is the group of isometries of
$H(R^{m}),$ together with its direct limit $\mathrm{O}(R)=$ \textrm{Colim}%
$\mathrm{O}_{m,m}(R).$ For $n\geq1,$ the higher hermitian $K$-groups are
defined in a way parallel to higher $K$-groups, using Quillen's +
construction, by the formula%
\[
KQ_{n}(R)=\pi_{n}(B\mathrm{O}(R)^{+}).
\]
However, from classical group considerations, as we already have seen, the
spinorial group behaves better than the orthogonal group for our purposes.
Therefore, we shall replace the group $\mathrm{O}_{m,m}(R)$ by the associated
spinorial group \textrm{Spin}$_{\mathrm{m,m}}(R)$ defined in Section 1, which
direct limit is denoted by \textrm{Spin}$(R)$. We have the following two exact
sequences (where the first one splits):
\[
1\longrightarrow\mathrm{SO}^{0}(R)\longrightarrow\mathrm{O}(R)\longrightarrow
Z_{2}(R)\oplus\mathrm{Disc}(R)\longrightarrow1,
\]%
\[
1\longrightarrow\mu_{2}(R)\longrightarrow\mathrm{Spin}(R)\longrightarrow
\mathrm{SO}^{0}(R)\longrightarrow1.
\]
Using classical tools of Quillen's + construction \cite{Hausmann-Husemoller},
one can show that the maps \textrm{SO}$^{0}(R)\longrightarrow\mathrm{O}(R)$
and \textrm{Spin}$(R)\longrightarrow$\textrm{SO}$^{0}(R)$ induce isomorphisms%
\[
\pi_{n}(B\mathrm{SO}^{0}(R)^{+})\cong\pi_{n}(B\mathrm{O}(R)^{+})\text{ for
}n>1.
\]%
\[
\pi_{n}(B\mathrm{Spin}(R)^{+})\cong\pi_{n}(B\mathrm{SO}^{0}(R)^{+})\text{ for
}n>2.
\]
Moreover, the maps
\[
\pi_{1}(B\mathrm{SO}^{0}(R)^{+})\longrightarrow\pi_{1}(B\mathrm{O}(R)^{+})
\]
and
\[
\pi_{2}(B\mathrm{Spin}(R)^{+})\longrightarrow\pi_{2}(B\mathrm{SO}^{0}(R)^{+})
\]
are injective.

Our extension of the Bott class to higher hermitian $K$-groups will be a map
also called $\rho_{k}:$%
\[
\rho_{k}:\pi_{n}(B\mathrm{Spin}(R)^{+})\longrightarrow\pi_{n}(B\mathrm{GL}%
(R)^{+})\left[  1/k\right]  =K_{n}(R)\left[  1/k\right]
\]

In order to define such a map, we work geometrically, using the description of
the various $K$-theories in terms of flat bundles as detailed in the appendix
1 to \cite{Karoubi asterisque}. Any element of $\pi_{n}(B$\textrm{Spin}%
$(R)^{+})=K$\textrm{Spin}$_{n}(R)$ for instance is represented by a formal
difference $x=V-T,$ where $V$ and $T$ are flat spinorial bundles of the same
rank, say $2m,$ over a homology sphere $X=\widetilde{S}^{n}$ of dimension $n.$
We may also assume that the fibers of $V$ and $T$ are hyperbolic e.g.
$H(R^{m})$ and that $T$ is "virtually trivial", which means that $T$ is the
pull-back of a flat bundle over an acyclic space. More precisely, we should
first consider a flat principal bundle $Q$ of structural group \textrm{Spin}%
$_{\mathrm{m,m}}(R)$ such that%
\[
V=Q\times_{\mathrm{Spin}_{\mathrm{m,m}}(R)}H(R^{m}).
\]
On the other hand, \textrm{Spin}$_{\mathrm{m,m}}(R)$ acts on $C(H(R^{m}))=$
\textrm{End}$(\Lambda R^{m})$ by inner automorphisms. Therefore, the bundle of
Clifford algebras $C(V)$ associated to $V$ is the bundle of endomorphisms of
the flat bundle%
\[
Q\times_{\mathrm{Spin}_{\mathrm{m,m}}^{\prime}(R)}\Lambda R^{m}.
\]
We now apply our general recipe of Section 2 on each fiber of $V$ and $T$. In
other words, for the bundle $V$ for instance, we consider the following
composition%
\[
K(X)\overset{\alpha}{\longrightarrow}K^{C(V)}(X)\overset{\Psi^{k}%
}{\longrightarrow}K^{C(V(k))}(X)\overset{\alpha_{k}^{-1}}{\longrightarrow
}K(X).
\]

In this sequence, we write $K(Y)$ for the group of homotopy classes of maps
from $Y$ to the classifying space of algebraic $K$-theory which is
homotopically equivalent to $K_{0}(R)\times BGL(R)^{+}$. Its elements are
reprented by flat bundles over spaces $X$ homologically equivalent to $Y.$ The
notation $K^{C(V)}(X)$ means the (graded) $K$-theory of flat bundles provided
with a graded $C(V)$-module structure.

The image of $1$ by the composition $\alpha_{k}^{-1}.\Psi^{k}.\alpha$ defines
an element of $K(X),$ which we call $\rho_{k}(V).$ On the other hand, since
$T$ is virtually trivial, we have $\rho_{k}(T)=k^{m}$. We then define
$\rho_{k}(x)$ in the group $K_{n}(R)\left[  1/k\right]  $ by the formula
\[
\rho_{k}(x)=\rho_{k}(V)/k^{m}%
\]

\begin{theorem}
For $n\geq1,$ the correspondance $x\mapsto\rho_{k}(x)$ induces a group
homomorphism%
\[
\rho_{k}:K\mathrm{Spin}_{n}(R)\longrightarrow K_{n}(R)\left[  1/k\right]
\]
called the $n$-hermitian Bott class.
\end{theorem}

\begin{proof}
The map $x\mapsto\rho_{k}(x)$ is well-defined by general homotopy
considerations. In order to check that we get a group homomorphism, we write
the direct sum of the $K_{n}(R)\left[  1/k\right]  $ as the multiplicative
group%
\[
1+K_{\ast>0}(R)\left[  1/k\right]
\]
where the various products between the $K_{n}$-groups are reduced to $0.$ If
we now take two elements $x$ and $y$ in $K$\textrm{Spin}$_{n}(R),$ we write
$\rho_{k}(x)$ as $1+u$ and $\rho_{k}(y)$ as $1+v$. Then
\[
\rho_{k}(x+y)=\rho_{k}(x)\cdot\rho_{k}(y)=(1+u)\cdot(1+v)=1+u+v
\]
since $u\cdot v=0.$
\end{proof}

\section{Relation with Topology}

Let $V$ be a real vector bundle on a compact space $X$ provided with a
positive definite quadratic form. We assume that $V$ is spinorial of rank
$8n,$ so that the bundle of Clifford algebras may be written as \textrm{End}%
$(E),$ where $E$ is the \textbf{Z}$/2$-graded vector bundle of "spinors" (see
for instance the appendix and \cite{Donovan-Karoubi}). Following Bott
\cite{Bott}, we define the class $\rho_{top}^{k}(V)$ as the image of $1$ by
the composition of the homomorphisms%
\[
K_{\mathbf{R}}(X)\overset{\varphi}{\longrightarrow}K_{\mathbf{R}}%
(V)\overset{\psi^{k}}{\longrightarrow}K_{\mathbf{R}}(V)\overset{\varphi^{-1}%
}{\longrightarrow}K_{\mathbf{R}}(X),
\]
where $K_{\mathbf{R}}$ denotes real $K$-theory and $\varphi$ is Thom's
isomorphism for this theory. One purpose in this section is to show that this
topological class $\rho_{top}^{k}(V)$ coincides with our hermitian Bott class
$\rho_{k}(V)$ in the group $K_{\mathbf{R}}(X)\cong K(R),$ where $R$ is the
ring of real continuous functions on $X.$

The proof of this statement requires a careful definition of the group
$K_{\mathbf{R}}(V),$ since $V$ is not a compact space. A possibility is to
define this group as follows (see \cite[\S 2]{Karoubi livre} for instance).
One considers couples $(G,D),$ where $G$ is a \textbf{Z}$/2$-graded real
vector bundle on $V,$ provided with a metric, and $D$ an endomorphism of $G$
with the following properties:

1) $D$ is an isomorphism outside a compact subset of $V$ and we identify
$(G,D)$ to $0$ if this compact set is empty

2) $D$ is self-adjoint and of degree $1.$

One can show that the Grothendieck group associated to this semi-group of
couples is the reduced $K$-theory of the one-point compactification of $V.$
With this description, Thom's isomorphism
\[
K_{\mathbf{R}}(X)\overset{\varphi}{\longrightarrow}K_{\mathbf{R}}(V)
\]
is easy to describe. It associates to a vector bundle $F$ on $X$ the couple%
\[
\tau=(G,D)=(F\otimes E,1\otimes\rho(v)).
\]
In this formula, $C(V)\cong$ \textrm{End}$(E)$ and $\rho(v)$ denotes the
Clifford multiplication by $\rho(v)$ over a point $v\in V\subset C(V).$ We
note that $\rho(v)$ is an isomorphism outside the $0$-section of $V.$
Therefore, Thom's isomorphism may be interpreted as Morita's equivalence. In
fact, $\varphi$ is the following composition (where $K(C(V))$ is the
$K$-theory of the ring of sections of the algebra bundle $C(V))$%
\[
K_{\mathbf{R}}(X)\longrightarrow K(C(V))\overset{t}{\longrightarrow
}K_{\mathbf{R}}(V),
\]
according to \cite{Karoubi livre} for instance.

\begin{remark}
This class $\rho_{top}^{k}(V)$, which requires a metric and a spinorial
structure on $V,$ is different in general from the algebraic class $\rho
^{k}(V)$, defined for $\lambda$-rings$.$ On the other hand, the quotient
between $\rho_{top}^{k}(V)$ and $(\rho^{k}(V))^{2}$ is a $2$-torsion class
which is not trivial in general, as it is shown in an example at the end of
this Section.
\end{remark}

If we apply the Adams operation to the previous couple $\tau=(G,D)$, one finds
$(\Psi^{k}(F)\cdot\Psi^{k}(E),\Psi^{k}(1\otimes\rho(v)))$ with obvious
definitions. Strictly speaking, $\Psi^{k}(E)$ should be thought of as a
virtual module over $C(V^{k})$ and then we use the diagonal $V\longrightarrow
V^{k}$ in order to view $\Psi^{k}(E)$ as a virtual module over $C(V)$. We use
here the functorial definition of the Adams operation detailed in the proof of
Theorem $2.5.$ Moreover, since $k$ has a square root as a positive real
number, we always have $C(V,kq)\cong C(V,q).$ To sum up, we have proved the
following theorem:

\begin{theorem}
Let $R$ be the ring of real continuous functions on a compact space $X$ and
let $V$ be a real spinorial bundle of rank $8n$ on $X.$ Then the topological
Bott class $\rho_{top}^{k}(V)$ in $K(R)$ coincides with the hermitian Bott
class $\rho_{k}(V)$ of $V,$ viewed as a finitely generated projective module
provided with a positive definite quadratic form and a spinorial structure.
\end{theorem}

For completeness' sake, let us make some explicit computations of this
hermitian Bott class when $X$ is a sphere of dimension $8m.$ Let $V$ be a real
oriented vector bundle of rank $4t$ on $S^{8m},$ generating the reduced real
$K$-group $\widetilde{K}_{\mathbf{R}}(S^{8m})$ and let $W=V\oplus V$ be its
complexification which generates $\widetilde{K}_{\mathbf{C}}(S^{8m}),$ where
$\widetilde{K}_{\mathbf{C}}$ is reduced complex $K$-theory. Let us denote by
$W_{\mathbf{R}}$ the underlying real vector bundle with the associated
spinorial structure. According to \cite[Proposition 7.27]{Karoubi livre}, we
have the formula
\[
c(\rho_{top}^{k}(W_{\mathbf{R}}))=\rho_{\mathbf{C}}^{k}(W)=\rho^{k}(W),
\]
where $c:K_{\mathbf{R}}(S^{8m})\overset{\cong}{\longrightarrow}K_{\mathbf{C}%
}(S^{8m})$ denotes the complexification. Therefore, we are reduced to
computing the class $\rho^{k}$ for complex vector bundles on even dimensional
spheres $X.$

If $X=S^{2}$, $K_{\mathbf{C}}(S^{2})$ is free of rank 2, generated by $1$ and
the Hopf line bundle $L.$ The classical Bott class is then computed from the
formula%
\[
\rho^{k}(L)=1+L+...+L^{k-1}.
\]
Since $(L-1)^{2}=0,$ another way to write this sum is to consider Taylor's
expansion of the polynomial%
\[
1+X+...+X^{k-1}%
\]
at $X=1.$ We get the formula%
\[
\rho^{k}(L)=k+\left[  1+2+...+(k-1)\right]  (L-1)=k+k(k-1)(L-1)/2.
\]
If $x_{2}$ denotes the class $L-1,$ we also can write
\[
\rho^{k}(x_{2})=1+\left[  1+2+...+(k-1)\right]  /k\cdot x_{2}.
\]
We compute in the same way the Bott class on $\widetilde{K}_{\mathbf{C}}%
(S^{4})\cong$ \textbf{Z}$,$ generated by the product%
\[
x_{4}=(L_{1}-1)\cdot(L_{2}-1)
\]
where $L_{1}$and $L_{2}$ are two copies of the Hopf line bundle on $S^{2}$.
Since we again have $(L_{i}-1)^{2}=0,$ it is sufficient to compute the first
terms of Taylor's expansion of the polynomial%
\[
f(X,Y)=1+XY+X^{2}Y^{2}+...+X^{k-1}Y^{k-1}%
\]
at the point $(1,1).$ We get the second derivative ($\delta^{2}f/\delta
x\delta y)/k^{2}$ at the point $(1,1)$ multiplied by $x_{4}.$ In other words,
we have%
\[
\rho^{k}(x_{4})=1+\left[  1+2^{2}+...+(k-1)^{2})\right]  /k^{2}\cdot x_{4}.
\]

More generally, on $\widetilde{K}_{\mathrm{C}}(S^{2r})\cong$ \textbf{Z}$,$
generated by
\[
x_{2r}=(L_{1}-1)\cdot\cdot\cdot(L_{r}-1),
\]
we find the formula%
\[
\rho^{k}(x_{2r})=1+\left[  1+2^{r}+...+(k-1)^{r}\right]  /k^{r}\cdot x_{2r}.
\]
Since $c(\rho^{k}(x_{8m}))=\rho_{top}^{k}(y_{8m}),$ where $y_{8m}$(resp
$x_{8m})$ generates $\widetilde{K}_{\mathbf{R}}(S^{8m})$ (resp. $\widetilde
{K}_{\mathbf{C}}(S^{8m}$)), we deduce from the last formula the following proposition.

\begin{proposition}
Let $V$ be a real vector bundle of rank $8t$ generating the real reduced
$K$-theory of the sphere $S^{8m}$ and let $y_{8m}=V-8t.$ We then have the
formula%
\[
\rho_{top}^{k}(y_{8m})=1+\left[  1+2^{4m}+...+(k-1)^{4m}\right]  \cdot
y_{8m}.
\]

\end{proposition}

\begin{remark}
If we assume that $k$ is odd, the sum $1+2^{r}+...+(k-1)^{r}$ has the same
parity as $1+2+...+(k-1)$ or equivalently $(k-1)/2\ $which is also odd for an
infinite number of odd $k^{\prime}s.$
\end{remark}

A more delicate example is the case of the sphere $X=S^{8m+2}$ with $m>0.$ It
is well known that the realification map%

\[
\mathrm{Z}\text{ }\mathbb{\cong}\text{ }\widetilde{K}_{\mathrm{C}}%
(S^{8m+2})\longrightarrow\widetilde{K}_{\mathrm{R}}(S^{8m+2})\cong\mathrm{Z}/2
\]
is surjective. Let $V$ be a complex vector bundle over $S^{8m+2}$ which
generates $\widetilde{K}_{\mathrm{C}}(S^{8m+2})$. We consider the following
diagram%
\[%
\begin{array}
[c]{ccc}%
K_{\mathrm{C}}(V) & \overset{\Psi^{k}}{\longrightarrow} & K_{\mathrm{C}}(V)\\
\uparrow\phi_{\mathrm{C}} &  & \downarrow\phi_{\mathrm{C}}^{-1}\\
K_{\mathrm{C}}(S^{8m+2}) &  & K_{\mathrm{C}}(S^{8m+2})
\end{array}
,
\]
where $\phi_{\mathrm{C}}$ is Thom's isomorphism in complex $K$-theory. By
definition, we have%
\[
\rho^{k}(V)=\phi_{\mathrm{C}}^{-1}(\Psi^{k}(\phi_{\mathrm{C}}(1))).
\]
Since $m>0,V$ is also a spinorial bundle and we therefore have a commutative
diagram up to isomorphism%
\[%
\begin{array}
[c]{ccc}%
K_{\text{\textrm{C}}}(V) & \overset{r}{\longrightarrow} & K_{\text{\textrm{R}%
}}(V)\\
\uparrow\phi_{\text{\textrm{C}}} &  & \uparrow\phi_{\text{\textrm{R}}}\\
K_{\mathrm{C}}(S^{8m+2}) & \overset{r}{\longrightarrow} & K_{\text{\textrm{R}%
}}(S^{8m+2})
\end{array}
,
\]
where $\phi_{\text{\textrm{R}}}$ is Thom's isomorphism in real $K$-theory and
$r$ is the realification. Since the Adams operation $\Psi^{k}$ commutes with
$r$, we have the identity%
\[
\rho_{top}^{k}(y_{8m+2})=1+\left[  1+2^{4m+1}+...+(k-1)^{4m+1}\right]  \cdot
y_{8m+2}=1+y_{8m+2}%
\]
if $k$ and $(k-1)/2$ are odd.

Let $V_{0}$ be the underlying real vector bundle of $V.$ The last identity
implies that $\rho_{top}^{k}(V_{0})=(1+y_{8m+2})\cdot k^{4t}$ if $V_{0}$ is of
rank $8t.$ Therefore $\rho_{top}^{k}(V_{0})$ cannot be a square, even modulo
the Picard group (which is trivial in this case). This implies that the
corrected hermitian Bott class defined in \ref{Corrected Bott class}:%
\[
\overline{\rho}_{k}:W\mathrm{Spin}(R)\longrightarrow K(R)^{\times
}/(\mathrm{Pic}(R))^{(k-1)/2}=K(R)^{\times}%
\]
is not trivial either (we recall that $R$ is the ring of real continuous
functions on the sphere $S^{8m+2}).$

\bigskip

\begin{remark}
We should add a few words if $k$ is even. If $k=2$ for instance and if $X$ is
the sphere $S^{8n}$, we find that%
\[
\rho^{2}(y_{8n})=1/2^{4n}\cdot y_{8n}%
\]
We get the same result for bundles with negative definite quadratic forms.
Since the hyperbolic map
\[
K(C_{\mathrm{R}}(S^{8n}))\cong\mathrm{Z}\longrightarrow KQ(C_{\mathrm{R}%
}(S^{8n}))\cong\mathrm{Z}\mathbb{\oplus}\mathrm{Z}%
\]
is the diagonal, we see that the class $\rho_{2}$ of an hyperbolic module
belongs to $1/2^{4n-1}\mathrm{Z}.$ Therefore, at least for this example, the
class $\rho_{2}$ also detects non trivial Witt classes.
\end{remark}

\section{Oriented Azumaya algebras}

Another purpose of this paper is the extension of our definitions to Azumaya
algebras \cite{Auslander-G.}\cite{Bass Clifford}, beyond the example of
Clifford algebras. We first consider the non graded case.

\begin{defn}
Let $A$ be an Azumaya algebra. We say that $A$ is "oriented" if the
permutation of the two copies of $A$ in $A^{\otimes2}$ is given by an inner
automorphism associated to an element $\tau\in(A^{\otimes2})^{\times}$ of
order $2.$
\end{defn}

As a matter of fact, as it was pointed out to us by Knus andTignol, any
Azumaya algebra is oriented \footnote{However, the situation is different in
the \textrm{Z}/2-graded case as we shall show below.}. This is a theorem
quoted by Knus and Ojanguren\cite[Proposition 4.1, p.. 112.]{Knus-Ojanguren}
and attributed to O. Goldman. We shall illustrate it by a few typical examples.

The first easy but fundamental example is $A=$ \textrm{End}$(P),$ where $P$ is
a faithful finitely generated projective module. We identity $B=A\otimes A$
with \textrm{End}$(P\otimes P)$ and $B^{\times}$ with \textrm{Aut}$(P\otimes
P).$ The element $\tau$ required is simply the permutation of the two copies
of $P,$ viewed as an element of $(A\otimes A)^{\times}=$ \textrm{Aut}%
$(P\otimes P),$ as it can be shown by a direct computation.

Let now $D$ be a division algebra over a field $F$. We claim that $D$ is also
oriented. In order to show this, we consider the tensor product $A=$
$D\otimes_{F}F_{1},$ where $F_{1}$ is a finite Galois extension of $F,$ such
that $A$ is $F$-isomorphic to a matrix algebra \textrm{M}$_{n}(F_{1}%
)=\mathrm{End}(F_{1}^{n})$ and is therefore oriented according to our first
example. Let $G$ be the Galois group of $F_{1}$ over $F,$ so that $D$ is the
fixed algebra of $G$ acting on $A.$ If we compose this action by the usual
action of the Galois group on $\mathrm{M}_{n}(F_{1}),$ we get automorphisms of
$\mathrm{M}_{n}(F_{1})$ as a $F_{1}$-algebra which are inner by
Skolem-Noether's theorem. If $g\in G,$ we let $\alpha_{g}$ be an element of
$\mathrm{Aut}(F_{1}^{n})$ so that the action $\rho(g)$ of $g$ on $A$ is given
by the composition of the inner automorphism associated to $a_{g}$ with the
usual Galois action on $\mathrm{M}_{n}(F_{1}).$

Let now $\tau^{\prime}$ be the permutation of the two copies of $A$ in the
tensor product $A\otimes_{F_{1}}A.$ It is induced by the inner automorphism
associated to a specific element $\tau$ in \textrm{Aut}$(F_{1}^{n}%
\otimes_{F_{1}}F_{1}^{n})$ of order $2$ which commutes with $\rho
(g)\otimes\rho(g)$ Therefore, $\tau$ is invariant by the action of $G$ and
belongs to $(D\otimes_{F}D)^{\times},$ considered as a subgroup of
$(A\otimes_{F_{1}}A)^{\times}.$

From a different point of view, let us consider the algebra $R$ of complex
continuous functions on a connected compact space $X.$ According to a
well-known dictionnary of Serre and Swan, one may consider an Azumaya algebra
$A$ over $R$ as a bundle $\widetilde{A}$ of algebras over $X$ with fiber
\textrm{End}$(P),$ where $P=$\textbf{C}$^{n}$. The structural group of this
bundle is the projective linear group \textrm{Aut}$(P)/$\textbf{C}$^{\times}.$
In the same way, the structural group of $A^{\otimes2}$ is \textrm{Aut}%
$(P\otimes P)/$\textbf{C}$^{\times}.$ Therefore, the inner automorphism of
$A^{\otimes2},$ permuting the two copies of $A,$ is induced by the permutation
of the two copies of $P$. This is well defined globally since this permutation
commutes with the transition functions of $\widetilde{A}$.

Let $A$ be any Azumaya algebra. We would like to lift the action $\sigma_{k}$
of the symmetric group $S_{k}$ on $A^{\otimes k}$ to ($A^{\otimes k}%
$)$^{\times},$ such that we have a commutative diagram%
\[%
\begin{array}
[c]{ccc}
&  & (A^{\otimes k})^{\times}\\
& \overset{\widetilde{\sigma}_{k}}{\nearrow} & \downarrow\gamma\\
S_{k} & \overset{\sigma_{k}}{\longrightarrow} & \mathrm{Aut}(A^{\otimes k})
\end{array}
,
\]
where $\widetilde{\sigma}_{k}$ is a group homomorphism and $\gamma$ induces
inner automorphisms. This program is achieved in the "Book of Involutions"
\cite[Proposition 10.1, pg. 115.]{Book of involutions}, using again the
"Goldman element" quoted above. For completeness' sake, we shall sketch a
proof below, since we shall need it in the graded case too.

In order to define $\widetilde{\sigma}_{k}$, we use the classical description
of the symmetric group in terms of generators $\tau_{i}=(i,i+1),i=1,...,k-1,$
with the relations $(\tau_{i})^{2}=1,\tau_{i}\tau_{i+1}\tau_{i}=\tau_{i+1}%
\tau_{i}\tau_{i+1}$ and $\tau_{i}\tau_{j}=\tau_{j}\tau_{i}$ if $\left[
i-j\right\vert >1.$ Since $A$ is oriented, we may view the $\tau_{i}$ in
($A^{\otimes k}$)$^{\times}$ as the tensor product of $\tau$ by the
appropriate number of $1=Id_{A}.$ We easily check the previous relations,
except the typical one%
\[
\tau_{1}\tau_{2}\tau_{1}=\tau_{2}\tau_{1}\tau_{2}.
\]
(one may replace the couple $(1,2)$ by $(i,i+1)).$ However, we already have
$\tau_{1}\tau_{2}\tau_{1}=\lambda\tau_{2}\tau_{1}\tau_{2},$ where $\lambda\in
R^{\times}.$ The identity
\[
(\tau_{1}\tau_{2}\tau_{1})^{2}=(\tau_{2}\tau_{1}\tau_{2})^{2}=1
\]
also implies that $(\lambda)^{2}=1.$ The solution to our lifting problem is
then to keep the $\tau_{i}$ for $i$ odd and replace the $\tau_{i}$ for $i$
even by $\lambda\tau_{i},$ in order to get the required relations among the
$\tau^{\prime}s.$

The previous considerations may be translated in the framework of
\textbf{Z}$/2$-graded Azumaya algebras \cite[pg. 160]{Bass Clifford}. In this
case, we must require the element $\tau$ in the definition to be of degree
$0.$ Unfortunately, in general, a Clifford algebra is not oriented in the
graded sense. As a counterexample, we may choose $A=C^{0,1}.$ Then the
permutation of the two copies of $A$ in $A\otimes A=C^{0,2}$ is given by the
inner automorphism associated to $e_{1}+e_{2}$ which is of degree $1$ and not
of degree $0,$ as required in our definition. However, if $V$ is a module
which is oriented and of even rank, the associated Clifford algebra $C(V)$ is
oriented as we shall show below.

\smallskip

To start with, let $V$ and $V^{\prime}$ be two quadratic modules such that $V$
is of even rank and oriented. The argument used in Section $1$ shows the
existence of an element $v$ in $C^{0}(V)\otimes C^{0}(V^{\prime})\subset
C(V)\widehat{\otimes}C(V^{\prime})\cong C(V\oplus V^{\prime})$ which
anticommutes with the elements of $V$ and commutes with the elements of
$V^{\prime}:$ one puts $v=u\otimes1$ with the notations of Section $1$ (see
Remark \ref{Definition of u}). Moreover, $(v)^{2}=1$ and $v\in$\textrm{Spin}%
$(V\oplus V^{\prime}).$

Let us choose $V^{\prime}=V$ and put $T=V\oplus V.$ Since $2$ is invertible in
$R,$ $T$ is isomorphic to the orthogonal sum $T_{1}\oplus T_{2},$ where
$T_{1}=\left\{  v,-v\right\}  $ and $T_{2}=\left\{  v,v\right\}  .$ Thanks to
this isomorphism, the permutation of the two summands of $V\oplus V$ is
translated into the involution $(t_{1},t_{2})\mapsto(-t_{1},t_{2})$ on
$T_{1}\oplus T_{2}$. Therefore, the previous argument shows the existence of a
canonical element $u_{12}\in$ \textrm{Spin}$(V\oplus V)$ of square $1$ such
that the transformation%
\[
x\mapsto u_{12}^{-1}.x.u_{12}%
\]
permutes the two summands of $V\oplus V.$ It follows immediately that the
Clifford algebra $C(V)$ is oriented (in the graded sense) if $V$ is oriented
and of even rank.

From the previous general considerations, we deduce a natural representation
of the symmetric group $S_{k}$ in the group \textrm{Spin}$(V^{k})$ which lifts
the canonical representation of $S_{k}$ in \textrm{SO}$(V^{k}).$ To sum up, we
have proved the following Theorem:

\begin{theorem}
Let $V$ be a quadratic module of even rank which is oriented and let%
\[
\sigma_{k}:S_{k}\longrightarrow\mathrm{SO}(V^{k})
\]
be the standard representation. Then there is a canonical lifting%
\[
\widetilde{\sigma}_{k}:S_{k}\longrightarrow\mathrm{Spin}(V^{k}),
\]
such that the following diagram commutes%
\[%
\begin{array}
[c]{ccc}
&  & \mathrm{Spin}(V^{k})\\
& \overset{\widetilde{\sigma}_{k,}}{\nearrow} & \downarrow\pi\\
S_{k} & \overset{\sigma_{k}}{\longrightarrow} & \mathrm{SO}(V^{k})
\end{array}
.
\]
In other words, the Clifford algebra $C(V)$ is a \textbf{Z}$/2$-graded
oriented Azumaya algebra.
\end{theorem}

\begin{Remarks}
One can also make an explicit computation in the Clifford algebra $C(V^{k})$
with the obvious elements $\tau_{i}=u_{i,i+1}.$; one checks they satisfy the
required relations for the generators of the symmetric group $S_{k}.$
Moreover, these liftings for various $k^{\prime}s$ are of course compatible
with each other. If $R$ is an integral domain, we note that $\widetilde
{\sigma}_{k}$ is unique, once $\widetilde{\sigma}_{2}$ is given.
\end{Remarks}

\section{Adams operations revisited}

In this Section we assume that $V$ is a quadratic $R$-module which is oriented
and of even rank, so that the Clifford algebra is a \textrm{Z}/2-graded
oriented Azumaya algebra.

If $k!$ is invertible in $R$ we have defined Adams operations in a functorial
way:%
\[
\Psi^{k}:K(C(V))\longrightarrow K(C(V(k))).
\]
The purpose of this Section is to define similars operation $\overline{\Psi
}^{k}$ under another type of hypothesis: $2k$ is invertible in $R$ and $R$
contains the ring of integers in the $k$-cyclotomic extension of $\mathbb{Q}$
which is
\[
\Omega_{k}=\mathbf{Z}(\omega)=\mathbf{Z}\left[  x\right]  /(\Phi_{k}(x)).
\]
Here $\Phi_{k}(x)$ is the cyclotomic polynomial and $\omega$ is the class of
$x$. We conjecture that $\overline{\Psi}^{k}=\Psi^{k}$ (which is defined via
Newton polynomials from the $\lambda$-operations) but we are not able to prove
it, except when $k!$ is invertible in $R.$ We also want to extend these
operations $\Psi^{k}$ and $\overline{\Psi}^{k}$ to oriented Azumaya algebras
(not only Clifford algebras) which were defined in the previous Section.

\smallskip

The idea to define $\overline{\Psi}^{k}$ is a remark by Atiyah \cite[Formula
$2.7$]{Atiyah} (used already in Section 3) that Adams operations may be
defined using the cyclic group \textbf{Z}$/k$ instead of the symmetric group
$S_{k}$ (if $k!$ is invertible in $R).$ More precisely, the Adams operation
$\Psi^{k}$ is induced by the homomorphism $R(S_{k})\longrightarrow$ \textbf{Z}
which associates to a representation $\sigma$ its character on the cycle
$(1,2...,k).$ Therefore, if we put $F=E^{\otimes k},$ we see that
\[
\Psi^{k}(E)=%
{\textstyle\sum\limits_{j=0}^{k-1}}
F_{\omega^{j}}\cdot\omega^{j},
\]
where $\omega$ is a primitive $k^{th}$-root of unity and where $F_{\omega^{j}%
}$is the eigenmodule corresponding to the eigenvalue $\omega^{j}.$ The
previous sum belongs in fact to the subgroup $K(C(V(k))$ of $K(C(V(k))\otimes
_{\mathbf{Z}}\Omega_{k}.$

If we only assume that $k$ is invertible in $R,$ we can consider the previous
sum as a new operation. More precisely, we define%
\[
\overline{\Psi}^{k}(E)=%
{\textstyle\sum\limits_{j=0}^{k-1}}
F_{\omega^{j}}\cdot\omega^{j}.
\]
In this new setting, this sum belongs to the group $K(C(V(k)))\otimes
_{\mathbf{Z}}\Omega_{k}$ and not necessarily to the subgroup $K(C(V(k)))$.
This definition makes sense since we have assumed $k$ invertible in $R,$ so
that $F$ splits as the direct sum of the eigenmodules associated to the
eigenvalues $\omega^{j}$, where $0\leq j\leq k-1.$ Since we work in the
\textbf{Z}$/2$-graded case, we also have to assume that $2$ is invertible in
$R.$

If $k$ is prime, because of the underlying action of the symmetric group on
$F,$ the eigenmodules $F_{\omega^{j}}$ are isomorphic to each other when
$1\leq j\leq k-1,$ so that this definition of $\overline{\Psi}^{k}(E)$ reduces
to $F_{0}-F_{\omega}.$We may be more precise and choose as a model of the
symmetric group $S_{k}$ the group of permutations of the set \textbf{Z}$/k.$
One generator $T$ of the cyclic group \textbf{Z}$/k$ is the permutation
$x\mapsto x+1.$ If $\alpha$ is a generator of the multiplicative cyclic group
$($\textbf{Z}$/k)^{\times},$ the permutation $x\longmapsto\alpha^{s}x,$ where
$s$ runs from $1$ to $k-2,$ enables us to identify all the eigenmodules
$F_{\omega^{j}},j=2,..,k-1$ with $F_{\omega}.$ We therefore get the following theorem:

\begin{theorem}
Let $E$ be a graded $C(V)$-module and let us assume that $2k$ is invertible in
$R$ and that the $k^{th}$-roots of unity belong to $R$. We define
$\overline{\Psi}^{k}(E)$ in the group $K(C(V(k)))\otimes_{\mathbf{Z}}%
\Omega_{k}$ by the following formula
\[
\overline{\Psi}^{k}(E)=%
{\textstyle\sum\limits_{j=0}^{k-1}}
F_{\omega^{j}}\cdot\omega^{j}.
\]
If $E_{0}$ and $E_{1}$ are two such modules, we have
\[
\overline{\Psi}^{k}(E_{0}\otimes E_{1})=\overline{\Psi}^{k}(E_{0}%
)\cdot\overline{\Psi}^{k}(E_{1})
\]
in the Grothendieck groups $K(C(V(2k)))\otimes_{\mathbf{Z}}\Omega_{k}.$
Moreover, if $k$ is prime, $\overline{\Psi}^{k}(E)$ belongs to
$K(C(V(k)))\subset K(C(V(k)))\otimes_{\mathbf{Z}}\Omega_{k}$ and we have the
following formula in $K(C(V(k):$%
\[
\overline{\Psi}^{k}(E_{0}\oplus E_{1})=\overline{\Psi}^{k}(E_{0}%
)+\overline{\Psi}^{k}(E_{1})\text{.}%
\]
Finally, the operation $\overline{\Psi}^{k}$ coincides with the usual Adams
operation $\Psi^{k}$ if $k!$ is invertible in $R.$
\end{theorem}

\begin{proof}
When $k$ is prime, we have the isomorphism
\[
(E_{0}\oplus E_{1})^{\otimes k}\cong(E_{0})^{\otimes k}\oplus(E_{1})^{\otimes
k}\oplus\Gamma,
\]
where $\Gamma$ is a module of type $(H)^{k}$ with an action of $S_{k}$
permuting the factors of $(H)^{k}$. From elementary algebra, we see that
$\Gamma$ is not contributing to the computation of $\overline{\Psi}^{k}%
(E_{0}\oplus E_{1}),$ hence the second formula.

For the first formula, we compute $\overline{\Psi}^{k}(E_{0}\otimes E_{1})$ by
looking formally at the eigenmodules of $T\otimes T$ acting on $(E_{0}%
)^{\otimes k}\otimes(E_{1})^{\otimes k},$ considered as a module over
$C(V(k))\otimes C(V(k)).$ They are of course associated to the eigenvalues
$\omega^{i}\otimes\omega^{j}=\omega^{i+j}.$ Using the remark above, we can
write%
\begin{align*}
\overline{\Psi}^{k}(E_{0}\otimes E_{1})  &  =%
{\textstyle\sum\limits_{r=0}^{k-1}}
\left[  (E_{0}\otimes E_{1})^{\otimes k}\right]  _{r}\cdot\omega^{r}\\
&  =%
{\textstyle\sum\limits_{r=0}^{k-1}}
\text{ }%
{\textstyle\sum\limits_{i+j=r}}
\left[  (E_{0})^{\otimes k}\right]  _{i}\cdot\omega^{i}\cdot\left[
(E_{1})^{\otimes k}\right]  _{j}\cdot\omega^{j}=\overline{\Psi}^{k}%
(E_{0})\cdot\overline{\Psi}^{k}(E_{1}).
\end{align*}
Finally, for $k!$ invertible in $R,$ the fact that $\overline{\Psi}^{k}%
=\Psi^{k}$ is just the remark made by Atiyah \cite{Atiyah} quoted above.
\end{proof}

\smallskip

Let now $A$ be any \textbf{Z}$/2$-graded Azumaya algebra which is oriented. We
would like to define operations on the $K$-theory of $A$ of the following type%
\[
K(A)\longrightarrow K(A^{\otimes k}).
\]
For this, we again follow the scheme defined by Atiyah \cite[Formula $2.7$%
]{Atiyah} (if $k!$ is invertible in $R).$ The only point which requires some
care is the definition of the "power map"%
\[
K(A)\longrightarrow K_{S_{k}}(A^{\otimes k})=K(A^{\otimes k})\otimes
R(S_{k}).
\]
A priori, the target of this map is the $K$-group of the cross-product algebra
$S_{k}\ltimes A^{\otimes k}.$ However, as we have seen in the previous
Section, the representation of $S_{k}$ in \textrm{Aut}$(A^{\otimes k})$ lifts
as a homomorphism from $S_{k}$ to $(A^{\otimes k})^{\times}.$ Therefore, this
cross product algebra is the tensor product of the group algebra
\textbf{Z}$\left[  S_{k}\right]  $ with $A^{\otimes k}$.

Therefore, any homomorphism%
\[
\lambda:R(S_{k})\longrightarrow\mathbf{Z}%
\]
gives rise to an operation%
\[
\lambda_{\ast}:K(A)\longrightarrow K(A^{\otimes k}),
\]
as we showed in Section $3$.\ However, one has to be careful that this
operation depends on the orientation chosen on $A,$ i.e. on the lifting of the
representation $\sigma_{k}:S_{k}\longrightarrow$\textrm{Aut}$(A^{\otimes k})$
to a representation $\widetilde{\sigma}_{k}:S_{k}\longrightarrow(A^{\otimes
k})^{\times}$, in such a way that the diagram%
\[%
\begin{array}
[c]{ccc}
&  & (A^{\otimes k})^{\times}\\
& \overset{\widetilde{\sigma}_{k}}{\nearrow} & \downarrow\\
S_{k} & \overset{\sigma_{k}}{\longrightarrow} & \mathrm{Aut}(A^{\otimes k})
\end{array}
\]
commutes. If $R$ is an integral domain, this lifting is defined up to the sign
representation. However, in this case, we get a canonical choice of
$\widetilde{\sigma}_{k}$ as follows. Let $F$ be the quotient ring of $R$ and
$\overline{F}$ its algebraic closure. If we extend the scalar to $\overline
{F}$, $A$ becomes a matrix algebra \textrm{End}$(E)$ over $\overline{F},$ in
which case $(A^{\otimes k})^{\times}$ is identified with \textrm{Aut}%
$(E^{k}).$ We then choose the sign of the lifting $\widetilde{\sigma}_{k}$ in
such a way that it corresponds to the canonical lifting $S_{k}\longrightarrow
$\textrm{Aut}$(E^{k})$ by extension of the scalars.

Let us be more explicit and define the $k^{th}$-exterior power $\lambda
^{k}(M)$ of $M$ as an $A^{\otimes k}$-module in our setting. We take the
quotient of $M^{\otimes k}$ by the usual relations (where the $m_{i}$ are
homogeneous elements):%
\[
m_{s(1)}\otimes m_{s(2)}\otimes...\otimes m_{s(k)}=\varepsilon(s)\widetilde
{\sigma}_{k}(s)\deg(m_{s})m_{1}\otimes m_{2}\otimes...\otimes m_{k}.
\]
Here $\varepsilon(s)$ is the signature of the permutation $s$, $\widetilde
{\sigma}_{k}$ the lifting defined above and $\deg(m_{s})$ the signature of the
representation $s$ restricted to elements of odd degree. We note that
$\lambda^{k}(M)$ is a graded module over $A^{\otimes k}.$

\textbf{Example.} Let $A=$ \textrm{End}$(E)$ with the trivial grading and let
$\widetilde{\sigma}_{k}$ be the canonical lifting. Then, by Morita
equivalence, all left $A$-modules $M$ may be written as $E\otimes N,$ where
$N$ is an $R$-module. It is then easy to see that $\lambda^{k}(M)\cong
E^{\otimes k}\otimes\lambda^{k}(N),$ where $\lambda^{k}(N)$ is the usual
$k^{th}$-exterior power over the commutative ring $R,E^{\otimes k}$ being
viewed as a module over $A^{\otimes k}\cong$ \textrm{End}$(E^{\otimes k}).$ We
note that if we change the sign of the orientation, we get the symmetric power
$E^{\otimes k}\otimes S^{k}(N)$ instead of the exterior power.

It is convenient to consider the full exterior algebra $\Lambda(M)$ of $M$
which is the direct sum of all the $\lambda^{k}(M).$ As usual, $\Lambda(M)$ is
the solution of a universal problem. If $g:M\longrightarrow C$ is an
$R$-module map where $C$ is an $R$-algebra and if
\[
g(m_{s(1)})g(m_{s(2)}),...g(m_{s(k)})=\varepsilon(s)\widetilde{\sigma}%
_{k}(s)\deg(m_{s})s(m_{1})s(m_{2})...s(m_{k}),
\]
there is an algebra map $\Lambda(M)\longrightarrow C$ which makes the obvious
diagram commutative. If $M$ is a finitely generated projective $A$-module,
$\lambda^{k}(M)$ as a finitely generated projective $A^{\otimes k}$-module:
this is a consequence of the following theorem.

\begin{theorem}
Let $A$ be an oriented \textbf{Z}$/2$-graded Azumaya algebra and let $M$ and
$N$ be two finitely generated projective $A$-modules Then the exterior algebra
of $M\oplus N$ is canonically isomorphic to $\Lambda(M)\otimes_{R}\Lambda(N)$.
Moreover, in each degree $k,$ we get an isomorphism of $A^{\otimes k}$-modules.

\begin{proof}
The canonical map from $M\oplus N$ to $\Lambda(M)\otimes_{R}\Lambda(N)$
induces the usual isomorphism%
\[
\Lambda(M\oplus N)\longrightarrow\Lambda(M)\otimes_{R}\Lambda(N).
\]
In each degree $k,$ this map induces an isomorphism between $\lambda
^{k}(M\oplus N)$ and the sum of the $\lambda^{i}(M)\otimes_{R}\lambda
^{k-i}(N),$ viewed as $A^{\otimes k}$-modules.
\end{proof}
\end{theorem}

Following Grothendieck and Atiyah again, we define $\lambda$-operations on
$K$-groups:%
\[
\lambda^{k}:K(A)\longrightarrow K(A^{\otimes k})
\]
satisfying the usual identity%
\[
\lambda^{r}(M\oplus N)=%
{\textstyle\sum\limits_{k+l=r}}
\lambda^{k}(M)\cdot\lambda^{l}(N)
\]
as $A^{\otimes(k+l)}$-modules. We can also define the Adams operations by the
usual formalism.

We may view operations in this type of $K$-theory as compositions%
\[
K(A)\overset{P}{\longrightarrow}K(A^{\otimes k})\otimes_{\mathbf{Z}}%
R(S_{k})\overset{\theta}{\longrightarrow}K(A^{\otimes k}).
\]
Here $P$ is the power map defined through the lifting $\widetilde{\sigma}_{k}$
above. The second map $\theta$ is induced by an homomorphism $R(S_{k}%
)\longrightarrow$\textbf{Z}$.$ In particular, the Adams operation $\Psi^{k}$
is given by the homomorphism
\[
R(S_{k})\longrightarrow\mathbf{Z}\
\]
which associates to a representation $\rho$ its trace of the cycle
$(1,2,...,k).$

\begin{remark}
A careful analysis of these considerations shows that we don't need $k!$ to be
invertible in order to define the $\lambda$-operations in the non graded case.
However, we need $2$ to be invertible in the graded case and, moreover, $k!$
invertible in order to define the Adams operations with good formal properties.
\end{remark}

Another approach to the Adams operations, as we showed at the beginning of
this Section, only assumes that $2k$ is invertible in $R$ and that $R$
contains the $k^{th}$-roots of unity. If $E$ is a finitely generated
projective $A$-module, the tensor power $E^{\otimes k}$ is an $S_{k}\ltimes$
$A^{\otimes k}$ -module$.$ We can "untwist" the two actions of $S_{k}$ and
$A^{\otimes k},$ thanks to the orientation of $A$ and we end up with an
$A^{\otimes k}$-module $F,$ with an independant action of $S_{k}.$ We put
formally%
\[
\overline{\Psi}^{k}(E)=%
{\textstyle\sum\limits_{j=0}^{k-1}}
F_{j}\cdot\omega^{j}%
\]
where $F_{j}$ is the eigenmodule associated to the eigenvalue $\omega^{j}.$
The previous sum lies in $K(A^{\otimes k})\otimes_{\mathbf{Z}}\Omega_{k}$ and
even in the subgroup $K(A^{\otimes k})$ if $k$ is prime. This second
definition is very pleasant, since the formal properties of the Adams
operations can be checked easily with this formula (at least for $k$ prime).
We conjecture that $\Psi^{k}=\overline{\Psi}^{k}$ in this case too.

\section{Twisted hermitian Bott classes}

\bigskip\medskip

We are going to define more subtle operations, associated not only to the
$K$-theory of $A$ but also to the $K$-theory of $A\otimes B$, where $A=C(V)$
and $B=C(W)$ are two Clifford algebras. We no longer assume that $V$ and $W$
are of even rank or oriented. However, we assume $k$ odd, $2k$ invertible in
$R$ and that the $k^{th}$-roots of unity belong to $R.$ We also replace the
symmetric group $S_{k}$ by the cyclic group \textbf{Z}$/k$ in our previous
arguments. The reason for this change is the following remark. The natural
representation $\sigma_{k}:$ \textbf{Z}$/k\longrightarrow O(V^{k})$ has its
image in the subgroup \textrm{SO}$^{0}(V^{k})$ defined in Section 1 and lifts
uniquely to a representation of \textbf{Z}$/k$ in \textrm{Spin}$(V^{k}),$ so
that the following diagram commutes:%
\[%
\begin{array}
[c]{ccc}
&  & \mathrm{Spin}(V^{k})\\
& \nearrow & \downarrow\\
\mathbf{Z}/k & \longrightarrow & \mathrm{SO}^{0}(V^{k})
\end{array}
.
\]
This lifting does not exist in general for the symmetric group $S_{k}$, except
if $V$ is even dimensional and oriented, as we have seen in Section 5.

Let now $M$ be a finitely generated projective module over $A\otimes B=$
$C(V)\otimes C(W)=C(V\oplus W).$ We can compose the power map%
\[
K(A\otimes B)\longrightarrow K(\mathbf{Z}/k\ltimes(A\otimes B)^{\otimes
k})\cong K(\mathbf{Z}/k\ltimes C(V^{k}\oplus W^{k}))
\]
with the "half-diagonal"%
\begin{align*}
K(\mathbf{Z}/k\ltimes C(V^{k}\oplus W^{k}))  &  \longrightarrow K(\mathbf{Z}%
/k\ltimes C(V(k)\oplus W^{k}))\\
&  \cong K(C(V(k))\otimes(\mathbf{Z}/k\ltimes C(W^{k}))),
\end{align*}
as we did in Section $2$ for $W=0.$ From the considerations in Section $6,$ we
can "untwist" the action of \textbf{Z}$/k$ on the \textbf{Z}$/2$-graded
Azumaya algebra $C(W^{k}),$ so that \textbf{Z}$/k\ltimes C(W^{k})$ is
isomorphic to the usual group algebra \textbf{Z}$\left[  \mathbf{Z}/k\right]
\otimes C(W^{k}).$ Using the methods of Section $2$ and of the previous
Section$,$ we get a more precise power map:%
\[
K(C(V)\otimes C(W))\longrightarrow K(C(V(k))\otimes C(W^{k}))\otimes
R(\mathbf{Z}/k).
\]
Therefore, according to Atiyah again \cite{Atiyah}, any homomorphism
\[
\lambda:R(\mathbf{Z}/k)\longrightarrow\Omega_{k}%
\]
gives rise to a "twisted operation"%
\[
\lambda_{\ast}:K(C(V)\otimes C(W))\longrightarrow K(C(V(k))\otimes
C(W^{k}))\otimes\Omega_{k}.
\]
We apply this formalism to $W=V(-1),$ in which case $C(W)$ is the (graded)
opposite algebra of $C(V).$ Therefore, $K(C(V)\otimes C(W))\cong K(R)$ by
Morita equivalence. If we choose for $\lambda$ the map above, we define the
"twisted hermitian Bott class" as the image of $1$ by the composition%
\[
K(R)\cong K(C(V)\otimes C(W))\longrightarrow K(C(V(k))\otimes C(W^{k}))\otimes
R(\mathbf{Z}/k)
\]%
\[
\longrightarrow K(C(V(k))\otimes C(W^{k})\otimes\Omega_{k}.
\]
We have proved the following theorem:

\begin{theorem}
Let $V$ be an arbitrary quadratic module. The "twisted hermitian Bott class"
$\rho_{k}(V)$ belongs to the following group%
\[
\rho_{k}(V)\in K(C(V(k))\otimes C(W^{k}))\otimes_{\mathrm{Z}}\Omega_{k}.
\]
It satisfies the multiplicative property%
\[
\rho_{k}(V_{1}\oplus V_{2})=\rho_{k}(V_{1})\cdot\rho_{k}(V_{2}),
\]
taking into account the identification of algebras:%
\[
C((V_{1}\oplus V_{2})(k))\otimes C(W_{1}^{k}\oplus W_{2}^{k})
\]%
\[
\cong\left[  C(V_{1}(k))\otimes C(W_{1}^{k}))\right]  \otimes\left[
C(V_{2}(k))\otimes C(W_{2}^{k})\right]  .
\]
$.$
\end{theorem}

\begin{remark}
Let $(V,E)$ be a spinorial module and let us identify the four \textrm{Z}%
/2-graded algebras $C(V),$ \textrm{End}$(E),$ $C(V(k))$ and $C(W)$. Then, by
Morita equivalence, we see that the twisted hermitian Bott class coincides
(non canonically) with the untwisted one.
\end{remark}

\begin{remark}
It is easy to show that $V^{4}$ is an orientable quadratic module which
implies by \ref{Definition of u} that the Clifford algebra $C(V^{4})$ is
isomorphic to its opposite. Let now $k$ be an odd square which implies that
$k\equiv1\operatorname{mod}8.$ Since $C(V(k))\cong C(V)$ and $C(W^{k})$ are
Morita equivalent to $C(W)$, the target group of the twisted hermitian Bott
class is isomorphic to%
\[
K(C(V)\otimes C(W))\otimes_{\mathbf{Z}}\Omega_{k}\cong K(R)\otimes
_{\mathbf{Z}}\Omega_{k}.
\]
This shows that we have a commutative diagram up to isomorphism%
\[%
\begin{array}
[c]{ccc}%
K\mathrm{Spin}(R) & \longrightarrow & KQ(R)\\
\downarrow &  & \downarrow\\
K(R)\left[  1/k\right]  ^{\times} & \longrightarrow & \left[  K(R)\otimes
_{\mathbf{Z}}\Omega_{k}\right]  \left[  1/k\right]  ^{\times}%
\end{array}
,
\]
where the vertical maps are defined by hermitian Bott classes, twisted and untwisted.
\end{remark}

Finally, as we did in Section 3, we can "correct" the twisted hermitian Bott
class by using the result of Serre about the square root of the classical Bott
class \cite{Serre}. More precisely, if $V$ is a self-dual module of dimension
$n$, there is an explicit class $\sigma_{k}(V)$ in $K(R)\otimes_{\mathbf{Z}%
}\Omega_{k}$ which only depends on the exterior powers of $V,$ such that%
\[
\sigma_{k}(V)^{2}=\delta^{(k-1)/2}\rho^{k}(V),
\]
with $\delta=(-1)^{n}\lambda^{n}(V).$ The corrected twisted hermitian Bott
class is then defined by the formula%
\[
\overline{\rho}_{k}(V)=\rho_{k}(V)(\sigma_{k}(V))^{-1},
\]
taking into account the fact that $K(C(V(k))\otimes C(W^{k}))$ is a module
over the ring $K(R).$ We have $\overline{\rho}_{k}(V)\in\pm($\textrm{Pic}%
$(R))^{(k-1)/2}$ if $V$ is hyperbolic\footnote{The sign ambiguity is
unavoidable, since $V$ is of arbitrary dimension.}, as we showed in Section
$3$. Therefore, the previous formula for $\overline{\rho}_{k}$ defines a
morphism also called $\overline{\rho}_{k}$, between the classical Witt group
$W(R)$ and twisted $K$-theory modulo $\pm($\textrm{Pic}$(R))^{(k-1)/2}$(as a
multiplicative group), more precisely
\[
\overline{\rho}_{k}:W(R)\longrightarrow\left[  K(C(V(k))\otimes C(W^{k}%
))\otimes_{\mathbf{Z}}\Omega_{k}\right]  \left[  1/k\right]  /^{\times}%
\pm(\mathrm{Pic}(R))^{(k-1)/2}.
\]

\begin{remark}
If $k\equiv1\operatorname{mod}4,$ we can multiply $\sigma_{k}(V)$ by the sign
$(-1)^{n(k-1)4},$ as we did in Section 3. If we apply this sign change, the
new corrected twisted hermitian Bott class takes its values in the group%
\[
\left[  K(C(V(k))\otimes C(W^{k}))\otimes_{\mathbf{Z}}\Omega_{k}\right]
\left[  1/k\right]  /^{\times}/(\mathrm{Pic}(R))^{(k-1)/2},
\]
without any sign ambiguity.
\end{remark}

\section{Appendix. A remark about the Brauer-Wall group}

The purpose of this appendix is to prove the following theorem which is also
found in \cite[Proposition $5.3$ and Corollary $5.4$]{Auslander-G.} for the
non graded case. It is added to this paper for completeness' sake with a
\textbf{Z}$/2$-graded variant.

\begin{theorem}
Let $R$ be a commutative ring. Let $A$ be an $R$-algebra which is projective,
finitely generated and faithful as an $R$-module. Let $P$ and $Q$ be faithful
projective finitely generated $R$-modules such that%
\[
A\otimes\mathrm{End}(P)\cong\mathrm{End}(Q).
\]
Then $A$ is isomorphic to some \textrm{End}$(E),$ where $E$ is also faithful,
projective and finitely generated. The same statement is true for
\textbf{Z}$/2$-graded algebras and modules if $2$ is invertible in $R$.
\end{theorem}

In order to prove the theorem, we need the following classical lemma:

\begin{lemma}
Let $P$ be a faithful finitely generated projective $R$-module. Then there
exists an $R$-module $Q$ such that $F=$ $P\otimes Q$ is free. Moreover, if $P$
is \textbf{Z}$/2$-graded and if $2$ is invertible in $R$, we may choose $Q$
such that $F=R^{2m}=$ $R^{m}\oplus R^{m},$ with the obvious grading.
\end{lemma}

\begin{proof}
(compare with \cite[pg. $14$]{Donovan-Karoubi} and \cite[Corollary $16.2$%
]{BassKT and stable}). Since any module $P$ of this type is locally the image
of a projection operator $p$ of rank $r>0,$ we can look at the "universal
example". This universal ring $R$ is generated by variables $p_{i}^{j}$ where
$1\leq i\leq n$ and $1\leq j\leq n$ such that the matrix $p=$ $(p_{i}^{j})$ is
idempotent of trace $r.$ According to \cite[p. $39$]{BassKT and stable}, since
$R$ is of finite stable range, the element $y=\left[  P\right]  -\left[
r\right]  $ is nilpotent in the Grothendieck group $K(R),$ say $y^{N}=0$ for
some $N.$ Let us now consider the element%
\[
x=r^{N-1}-r^{N-2}y+...+(-1)^{N-1}y^{N-1}.
\]
We have the identity $(r+y)Mx=M(r^{N}-(-1)^{N-1}y^{N})=Mr^{N}.$ Since the rank
of $x$ is $r^{N-1}>0$ and since the stable range of $R$ is finite, the element
$Mx$ in $K(R)$ is the class of a module $Q$ for sufficiently large $M.$ If
follows that $P\otimes Q$ is stably free and therefore free if $M$ is again
large enough. Finally, the case of \textbf{Z}$/2$-graded modules follows by
the same argument, considering graded $R$-modules as $R\left[  \mathbf{Z}%
/2\right]  $-modules.
\end{proof}

\begin{proof}
(of the theorem). Let us first consider the non graded case. Without
restriction of generality, we may assume that $A$ is of constant rank and that
$P$ and $Q$ are also of constant rank such that $A\otimes$\textrm{End}$(P)$ is
isomorphic to \textrm{End}$(Q).$ According to the previous lemma, we may also
assume that $P$ is free of constant rank, say $n.$ Therefore, we have an
algebra isomorphism%
\[
A\otimes M_{n}(R)\cong\mathrm{End}(Q).
\]
Let us now consider the fundamental idempotents in the matrix algebra
$M_{n}(R)$ defined by the diagonal matrices with all elements $=0$ except one
which is $1$. Thanks to the previous isomorphism, we may use these idempotents
to split $Q$ as the direct sum of $n$ copies of $E.$ Since the commutant of
$M_{n}(R)$ in $A\otimes M_{n}(R)$ is $A,$ it follows that the representation
of $A$ in \textrm{End}$(Q)$ is the orthogonal sum of $n$ copies of a
representation $\rho$ from $A$ to \textrm{End}$(E).$ From the previous algebra
isomorphism, we therefore deduce the required identity%
\[
A\cong\mathrm{End}(E).
\]
Finally, we make the obvious modifications of the previous argument in the
\textbf{Z}$/2$-graded case by writing the previous algebra isomorphism in the
form%
\[
A\widehat{\otimes}M_{2n}(R)\cong\mathrm{End}(Q),
\]
with the obvious grading on $M_{2n}(R).$ We use again the fundamental
idempotents in $M_{2n}(R)$ in order to split $Q$ as a direct sum $E^{n},$
where $E$ is \textbf{Z}$/2$-graded.
\end{proof}

\section{}

\end{document}